\documentclass{amsart}
\usepackage{amssymb}
\usepackage[utf8]{inputenc}
\usepackage{hyperref}
\usepackage{enumitem}
\usepackage{float}
\usepackage{comment}
\usepackage{mathtools}
\numberwithin{equation}{section}
\newcommand*{\mailto}[1]{\href{mailto:#1}{\nolinkurl{#1}}}
\makeatletter

\setenumerate{label=$(\roman*)$,itemsep=3pt,topsep=3pt}

\newtheorem{theorem}{Theorem}[section]
\newtheorem{proposition}[theorem]{Proposition}
\newtheorem{lemma}[theorem]{Lemma}
\newtheorem{corollary}[theorem]{Corollary}
\newtheorem{definition}[theorem]{Definition}
\newtheorem{hypothesis}[theorem]{Hypothesis}

\newtheorem{problem}[theorem]{Open Problem} 
\theoremstyle{remark}

\newenvironment{remark}[1][]{\refstepcounter{theorem}\par\medskip\noindent\textit{Remark~$\theexample. #1$} \rmfamily}{{\ }\hfill $\diamond$ \vspace{6pt}}

\DeclareMathOperator{\dom}{dom}
\DeclareMathOperator{\supp}{supp}

\newcommand{\R}{\mathbb{R}}

\newcommand{\N}{\mathbb{N}}
\newcommand{\C}{\mathbb{C}}

\newcommand{\g}{\gamma}
\renewcommand{\d}{\delta}
\renewcommand{\a}{\alpha}
\renewcommand{\b}{\beta}

\newcommand{\wt}{\widehat \theta}
\newcommand{\wphi}{\widehat \varphi}

\newcommand{\ACl}{{AC_{loc}((a,b))}}
\newcommand{\Ll}{{L^1_{loc}((a,b);dx)}}
\newcommand{\Lr}{{L^2((a,b); r(x)dx)}}

\newcommand{\no}{\notag}

\newcommand{\floor}[1]{\lfloor #1 \rfloor}

\newcommand{\np}{v}

\title[Finer limit circle/limit point classification]{Finer limit circle/limit point classification for Sturm--Liouville operators}

\author[M. Piorkowski]{Mateusz Piorkowski}
\address{Department of Mathematics, KU Leuven \\
Celestijnenlaan 200B, 3001 Leuven, Belgium}
\email{\href{mailto:Mateusz.Piorkowski@kuleuven.be}{Mateusz.Piorkowski@kuleuven.be}}
\urladdr{\url{https://sites.google.com/view/mateuszpiorkowski/home}}

\author[J.\ Stanfill]{Jonathan Stanfill}
\address{Department of Mathematics, The Ohio State University \\
100 Math Tower, 231 West 18th Avenue, Columbus, OH 43210, USA}
\email{\mailto{stanfill.13@osu.edu}}
\urladdr{\url{https://u.osu.edu/stanfill-13/}}

\date{\today}
\@namedef{subjclassname@2020}{\textup{2020} Mathematics Subject Classification}

\begin{document}

\begin{abstract}
In this paper we introduce an index $\ell_c \in \mathbb{N}_0 \cup \lbrace \infty \rbrace$ which we call the `regularization index' associated to the endpoints, $c\in\{a,b\}$, of nonoscillatory Sturm--Liouville differential expressions with trace class resolvents. This notion extends the limit circle/limit point dichotomy in the sense that $\ell_c~=~0$ at some endpoint if and only if the expression is in the limit circle case. In the limit point case $\ell_c>0$, a natural interpretation in terms of iterated Darboux transforms is provided. We also show stability of the index $\ell_c$ for a suitable class of perturbations, extending earlier work on perturbations of spherical Schr\"odinger operators to the case of general three-coefficient Sturm--Liouville operators. We demonstrate our results by considering a variety of examples including generalized Bessel operators, Jacobi differential operators, and Schr\"odinger operators on the half-line with power potentials.
\end{abstract}

\maketitle

\section{Introduction}

The well-known limit circle/limit point classification introduced by Weyl tells us how many boundary conditions are necessary to define a self-adjoint realization of a three-coefficient Sturm--Liouville differential expression given by
\begin{align*}
    \tau = \dfrac{1}{r(x)}\Big[-\dfrac{d}{dx} p(x)  \dfrac{d}{dx} + q(x) \Big] \ \text{ for a.e.~$x\in(a,b) \subseteq \R$.}
\end{align*}
(see Hypothesis \ref{HypoInt} for details). Limit point implies no boundary conditions are needed at the given endpoint, while limit circle implies one boundary condition is needed. Since the introduction by Weyl, the limit circle/limit point dichotomy of Sturm--Liouville expressions has been thoroughly studied- for recent treatments of Sturm--Liouville theory with encyclopedic references, we refer the interested reader to \cite{GNZ23} and \cite{Ze05}. 

The purpose of the present paper is to study an extension of the classic binary classification of limit circle/limit point in the case of nonoscillatory Sturm--Liouville differential expressions. This is achieved by means of the regularization index which gives a natural finer classification of certain limit point endpoints: those with finite regularization index lead to eigenvalues satisfying Weyl asymptotics (i.e., grow like $n^2$); those with infinite regularization index have eigenvalues satisfying growth quicker than $n$ (i.e., trace class resolvent); those for which the regularization index is undefined have nonempty essential spectrum and/or eigenvalues that grow slower than the above.

In order to define the regularization index for the endpoint $a$ $($resp., $b)$ we need to assume $\tau |_{(a,c)}$ $($resp., $\tau |_{(c,b)})$ has self-adjoint realizations with trace class resolvents. This spectral condition, essential in our work, turns out to be equivalent with the following simple integrability condition on the product of the principal and nonprincipal solutions near the respective endpoint (see Theorem \ref{TFAE}):  Assume that the equation $\tau f = \frac{1}{r}(-(p f')'+ qf) = \lambda f$ on $(a,b)$, is nonoscillatory at the endpoint $a$ $($resp., $b)$ for some $\lambda \in \mathbb R$, and that for some $c\in (a,b)$
\begin{align*}
    \int_a^c |u_a(x)\np_a(x)r(x)| dx < \infty \qquad \left(\text{resp., } \int_c^b |u_b(x)\np_b(x)r(x)| dx < \infty\right)
\end{align*}
where $u_a$, $\np_a$ $($resp., $u_b$, $\np_b$$)$ are any principal, resp.~nonprincipal solutions of $\tau f=\lambda f$ near the endpoint $x = a$ $($resp., $x=b)$. This integrability assumption then allows one to iteratively construct a spectral parameter power series (i.e., a Taylor series in the spectral parameter $z$) for solutions of the Sturm--Liouville problem $\tau y=zy$. This series can be equivalently viewed as a type of Born expansion, and we show that in a certain precise sense this series is well-behaved if and only if the aforementioned trace class condition is satisfied (see Remark \ref{RemarkBorn}). The regularization index at the singular endpoint $x=a$ (resp., $x=b$) is then defined by comparing the growth in $x$ of the coefficients of the power series of the principal and nonprincipal solutions as $x\to a^+$ (resp., $x\to b^-)$. See Definition~\ref{DefRI}.

The regularization index turns out to be well-behaved under Darboux transforms as seen in Theorem \ref{Thm:Darboux}. Hence, one important implication is that a finite regularization index allows one to quantify how far certain limit point endpoints are away from being Darboux transformed to a limit circle endpoint, which in turn can be regularized in the sense of Niessen and Zettl (see \cite{NZ92} and \cite[Thm. 8.3.1]{Ze05}), thus the appropriateness of the name regularization index.
In particular, our work extends the notion of regularization to include limit point nonoscillatory endpoints of Sturm--Liouville expressions with finite regularization index  that can be Darboux transformed (equivalently, transformed into Schr\"odinger form). As a corollary, we obtain Weyl eigenvalue asymptotics for this class of problems. For more information on Darboux transforms directly related to the current study, we refer to \cite{D78}, \cite{GT96} (see also \cite{KST_MN} and the extensive list of references therein) and \cite{B19}, \cite{G-UKM}, \cite{G-UMM} (in the context of exceptional orthogonal polynomials).

We also study the structure of the Weyl $m$-function under our regularization process. In particular, we obtain an alternative proof that the Weyl $m$-functions in the case of finite regularization index is in the subclass $N_\kappa^\infty$ consisting of generalized Nevanlinna--Herglotz functions with $\kappa = \floor{(\ell+1)/2}$ negative squares (where $\lfloor x \rfloor$ is the floor function and $\ell=\min\{\ell_a,\ell_b\}$), no nonreal poles, and the only generalized pole of nonpositive type at infinity. This extends some of the results for specific examples studied in the series of papers \cite{KT_JDE}--\cite{KT13}. In addition to these papers, for more information on singular Weyl $m$-functions we refer to \cite{FL10}, \cite{GZ06}, and \cite{KST_IMRN}.

\medskip

The notion of the regularization index has appeared (sometimes implicitly) in various previous works. The prime example of a Sturm--Liouville operator for which the index is particularly useful is the perturbed spherical Schr\"odinger operator (or Bessel operator). Some earlier works include the papers of Fulton \cite{F08}, Fulton-Langer \cite{FL10}, and Kurasov-Luger \cite{KL11}, where analytic perturbations where studied and the authors relied on the Frobenius method. Here the regularization index can be explicitly computed in terms of the roots of the corresponding indicial equation. Kostenko, Sakhnovich, and Teschl in a series of papers \cite{KST_Inv}--\cite{KT13} included nonanalytic perturbations and used methods more inline with the present paper. In fact, our perturbation result Theorem \ref{ThmPer} can be viewed as a natural generalization of the perturbative approach used in \cite[Lem.~2.2]{KST_Inv} for spherical Schr\"odinger operators. See Remark \ref{RemarkPer} for more details.

The first explicit definition of an index, seemingly equivalent to the one defined in the present paper, seems to originate from the work of Kaltenb\"ack and Woracek on canonical systems and Pontryagin spaces of entire functions in \cite{KW6,KW5}. This index, denoted by the greek letter $\Delta$, can be associated to certain canonical Hamiltonian systems which encompass the Sturm--Liouville operators treated here as a special case. Of particular significance is the paper of Winkler and Woracek \cite{WW14}, where an accessible method for computing the index $\Delta$ is provided and the paper of Langer and Woracek \cite{LW23}, where the special case of Sturm--Liouville operators is treated. In particular, Langer and Woracek use essentially the same recursion to define $\Delta$ as we do to define $\ell$, however subtle differences remain; see, for instance, Open Problem \ref{problemTheta} and Remark \ref{RemarkDelta}. We nonetheless believe both notions to be equivalent and fully agree with the assessment of the authors of \cite{WW14} that `limit point but finite index' is in many respects similar to the limit circle case, which becomes even more apparent through our regularization process. Interestingly, when the index is infinite, certain examples can still share properties of the limit circle case such as Weyl asymptotics- see Section \ref{MiePotential} which includes inverse quartic potentials.

We do not use the theory of canonical systems or Pontryagin spaces in our work, though we certainly believe that there are interesting connections to these areas which deserve further attention. We also avoid the notion of rigged (distributional) Hilbert spaces and supersingular perturbations (see \cite{DS2000}, \cite{K03}, \cite[App.~A]{KL11}, \cite{LN16}). Instead our proofs mainly rely on classical ODE-methods for absolutely continuous functions. This is motivated by the fact that our Definition~\ref{DefRI} of the regularization index relies exclusively on growth properties of classical solutions to $\tau y = zy$ near the endpoints, rather than on their $L^2$-integrability or membership in a rigged Hilbert space. 

\medskip     

The present paper is organized as follows:
\begin{itemize}[leftmargin=*]
    \item In Section \ref{secttracecondition} we give some background and introduce the main integrability, equivalently trace class, assumptions that will be used throughout the paper.
    \item In Section \ref{sectphi}  we construct a spectral parameter power series representation for the principal solution to the Sturm--Liouville problem and show its convergence in Proposition \ref{PropPhiEntire}.
    \item In Section \ref{secttheta} we start with the crucial Theorem \ref{TFAE} giving us equivalent characterizations of the trace class resolvent condition in terms of properties of the principal and nonprincipal solutions. We also list an array of consequences of Theorem \ref{TFAE} for the properties of the entire nonprincipal solution, in particular, Corollaries~\ref{CorThetaNorm} and \ref{CorTheta}.
    \item In Section \ref{Sect_RI} we introduce the regularization index in Definition~\ref{DefRI} and discuss some simple examples in Remark \ref{remarkexamples}. We then proceed to prove a stability result in Theorem \ref{ThmPer}, generalizing earlier work on perturbed spherical Schr\"odinger operators.
    \item In Section \ref{sectlclp} we relate the regularization index to the classic limit circle/limit point classification. This is done in Theorem \ref{Thm:N=0}. We then study the relationship between the regularization index $\ell_a$  and the index $\Delta$ used in \cite{LW23}, \cite{WW14}. This relationship is encapsulated in Open Problem \ref{problemTheta}.
    \item  In Section \ref{sectnatural} we introduce in Definition~\ref{DefThetaNorm} a useful choice of normalized system of entire solutions, which we label `naturally normalized system'. This notion plays an important role in the subsequent sections. We also give a more intuitive characterization of this normalization in Theorem \ref{Theorem2} (see in particular \eqref{gg}).
    \item In Section \ref{SectDarb} we introduce additional hypotheses to study the relationship between Darboux transforms and naturally normalized systems. The crucial result of this section is Corollary \ref{cor:cannorm} showing that the natural normalization is preserved under Darboux transforms.
    \item In Section \ref{SectSpec} we determine how the regularization index changes under a Darboux transform depending on the (non)principality property of the seed function (see Theorem \ref{Thm:Darboux}). We also illustrate how applying a series of Darboux transforms can be viewed as a regularization process, showing that Weyl asymptotics hold for problems with finite regularization indices.
    \item In Section \ref{SectWeylm} we use the results on Darboux transforms of the previous section to explicitly compute Weyl $m$-functions for problems with finite regularization indices. 
    \item In Section \ref{sectexamples} we demonstrate our results by considering a variety of examples including generalized Bessel operators, Jacobi differential operators, and Schr\"odinger operators on the half-line with power potentials. We also provide an example with an infinite regularization index for which Weyl asymptotics still holds (Mie-type potentials) and consider the Laguerre operator at $\infty$ for which our main hypothesis is not satisfied.
    \item Appendix \ref{AppendixA} contains certain technical proofs not included in the main text. 
\end{itemize}

We include a few open problems throughout the paper.

\section{A trace class integrability Condition}\label{secttracecondition}

We begin by recalling the typical integrability hypotheses that we will make throughout.

\begin{hypothesis} \label{HypoInt}
Let $(a,b) \subseteq \R$ and suppose that $p,q,r$ are $($Lebesgue\,$)$ measurable functions on $(a,b)$ 
such that the following items $(i)$--$(iii)$ hold: \\[1mm] 
$(i)$ \hspace*{1.1mm} $r>0$ a.e.~on $(a,b)$, $r\in\Ll$. \\[1mm] 
$(ii)$ \hspace*{.1mm} $p>0$ a.e.~on $(a,b)$, $1/p \in\Ll$. \\[1mm] 
$(iii)$ $q$ is real-valued a.e.~on $(a,b)$, $q\in\Ll$. 
\end{hypothesis}

Given Hypothesis \ref{HypoInt}, we study Sturm--Liouville operators associated with the general, 
three-coefficient differential expression
\begin{equation} \label{SL}
\tau = \dfrac{1}{r(x)}\Big[-\dfrac{d}{dx} p(x)  \dfrac{d}{dx} + q(x) \Big] \ \text{ for a.e.~$x\in(a,b) \subseteq \R$.}
\end{equation}
As such, the Wronskian of $f$ and $g$, for $f,g\in\ACl$, is defined by
\begin{equation*}
W(f,g)(x) = f(x)g^{[1]}(x) - f^{[1]}(x)g(x), \quad x \in (a,b),
\end{equation*}
with 
\begin{equation*}
y^{[1]}(x) = p(x) y'(x), \quad x \in (a,b),
\end{equation*}
denoting the first quasi-derivative of a function $y\in AC_{loc}((a,b))$.
In the following we will drop the \emph{a.e.} from equalities between functions in $L^1_{loc}$.

Let us now introduce maximal and minimal operators in $L^2((a,b); r(x)dx)$ associated with $\tau$ in the usual manner as follows. 

\begin{definition}
Assume Hypothesis \ref{HypoInt}. Given $\tau$ as in \eqref{SL}, the \emph{maximal operator} $T_{max}$ in $L^2((a,b); r(x)dx)$ associated with $\tau$ is defined by
\begin{align*}
&T_{max} f = \tau f,    
\\
& f \in \dom(T_{max})=\big\{g\in\Lr \, \big| \,g,pg'\in\ACl;    \\ 
& \hspace*{6.3cm}  \tau g\in\Lr\big\}.  
\end{align*}
The \emph{preminimal operator} $T_{min,0} $ in $L^2((a,b); r(x)dx)$ associated with $\tau$ is defined by 
\begin{align*}
&T_{min,0}  f = \tau f,   \no
\\
&f \in \dom (T_{min,0})=\big\{g\in L^2((a,b); r(x)dx) \, \big| \, g,pg'\in\ACl;   
\\ 
&\hspace*{3.25cm} \supp \, (g)\subset(a,b) \text{ is compact; } \tau g\in L^2((a,b); r(x)dx) \big\}.   
\end{align*}

One can prove that $T_{min,0} $ is closable, and one then defines the \emph{minimal operator} $T_{min}$ as the closure of $T_{min,0} $. We have that $(T_{min,0})^* = T_{max}$.
\end{definition}

It is known (see, e.g.,~\cite{NZ92}) that if the equation 
\begin{align}\label{EigenEq}
    \tau f = \lambda f, \qquad \lambda \in \R
\end{align}
is nonoscillatory near $a$ (resp., $b$), meaning that its solutions have finitely many zeros in a vicinity of the respective endpoint, then there exists an up to constant multiples unique solution $u_a$ (resp., $u_b$) of \eqref{EigenEq} satisfying 
\begin{align*}
\lim_{x \to a^+} \frac{u_a(x)}{\np_a(x)} = 0 \qquad \Big(\text{resp.,}\, \lim_{x \to b^-} \frac{u_b(x)}{\np_b(x)} = 0\Big)
\end{align*}
for any linearly independent solution $\np_a$ (resp., $\np_b$) of \eqref{EigenEq}. In this case $u_a$ (resp., $u_b$) is called the \emph{principal solution} of \eqref{EigenEq} at $a$ (resp., $b$), and $\np_a$ (resp., $\np_b$) is called a \emph{nonprincipal solution} of \eqref{EigenEq}. Note that the nonoscillatory condition \eqref{EigenEq} near $a$ (resp., $b$) is equivalent to the semiboundedness of one (hence any) self-adjoint realization of $\tau|_{(a,c)}$ (resp., $\tau|_{(c,b)}$).

We now come to the main spectral condition of the present paper. We say that $\tau$ satisfies the \emph{trace class property at $x = a$} (resp., at $x = b$) if and only if every self-adjoint realization $T$ of $\tau|_{(a,c)}$ (resp., $\tau|_{(c,b)}$) for some (hence any) $c \in (a,b)$ has trace class resolvent $(T-zI)^{-1}$ for some (hence any) $z$ in the resolvent set $\rho(T)$. The main goal of the present paper is to study the implications of the trace class property in the semibounded case:
\begin{hypothesis}\label{hypothesisTrace}
    Assume that $\tau$ satisfies the trace class property at $a$ $($resp., $b)$ and that self--adjoint realizations of $\tau|_{(a,c)}$ $($resp., $\tau|_{(c,b)})$ with $c \in (a,b)$ are semibounded.
\end{hypothesis}
 
As we will demonstrate, it is more practical to work instead with Hypothesis \ref{Hypothesis} stated below, as it is computationally more tractable. We will eventually prove in Theorem \ref{TFAE} that Hypothesis \ref{Hypothesis} is in fact equivalent to Hypothesis \ref{hypothesisTrace} at the respective endpoint, giving an easy criterion for the trace class property of $\tau$. 
\begin{hypothesis}\label{Hypothesis}
    Assume that the equation $\tau f = \lambda f$ is nonoscillatory at the endpoint $a$ $($resp., $b)$ for some $\lambda \in \mathbb R$, and that for some $c\in (a,b)$
    \begin{align}\label{Condition}
        \int_a^c |u_a(x)\np_a(x)r(x)| dx < \infty \qquad \left(\text{resp., } \int_c^b |u_b(x)\np_b(x)r(x)| dx < \infty\right)
    \end{align}
    where $u_a$, $\np_a$ $($resp., $u_b$, $\np_b$$)$ are any principal, resp.~nonprincipal solutions of \eqref{EigenEq} near the endpoint $x = a$ $($resp., $x=b)$.
\end{hypothesis}

It is important to clarify that as our analysis is local, we will mostly require condition \eqref{Condition} to hold only at one of the endpoints, which we conventionally take as $a$. In case of doubt, we will explicitly state that we require Hypothesis \ref{Hypothesis} at $ x= a$ (resp., at $x=b$). We now add some context for this hypothesis.

\begin{remark}\label{remarkcoefficients}
$(i)$ Notice that Hypothesis \ref{Hypothesis} holds at limit circle nonoscillatory endpoints for every $\lambda\in\R$ by definition. \\[1mm]
$(ii)$ We will show in Corollary \ref{CorHyp} that if Hypothesis \ref{Hypothesis} holds for some $\lambda\in\R$ it holds for all $\lambda\in\R$. Therefore, if $q\equiv0$ in \eqref{SL}, we can choose without loss of generality $\lambda=0$ in Hypothesis \ref{Hypothesis}, and one confirms that linearly independent solutions are given by $y_1(x)=1$ and $y_2(x)=\int_x^{x_0} \frac{dt}{p(t)}$ with $x_0\in(a,b)$. So we now consider two case distinctions:

$(a)$ Suppose $1/p\notin L^1((a,c);dx)$. Then $\lim_{x\to a^+}y_1(x)/y_2(x)=0$ so $y_1$ is principal and $y_2$ is nonprincipal.
Hence in this case Hypothesis \ref{Hypothesis} at $x=a$ is equivalent to assuming that (cf. \cite[Def. 7.1]{LW23})
\begin{equation}\label{casea}
\int_a^c \int_x^c \frac{dt}{p(t)}\, r(x) dx<\infty, \qquad \frac{1}{p}\notin L^1((a,c);dx), \qquad c\in(a,b),
\end{equation}
which in particular implies that $r\in L^1((a,c);dx)$.

$(b)$ Suppose $1/p\in L^1((a,c);dx)$. Then $y_3(x)=\int_a^x \frac{dt}{p(t)}$ exists, solves $\tau y = 0$, and satisfies $\lim_{x\to a^+}y_3(x)/y_1(x)=0$ so that $y_3$ is principal and $y_1$ is nonprincipal.
Hence in this case Hypothesis \ref{Hypothesis} is equivalent to (cf. \cite[Def. 8.1]{LW23})
\begin{equation}\label{caseb}
\int_a^c \int_a^x \frac{dt}{p(t)}\, r(x) dx<\infty,\quad  \frac{1}{p}\in L^1((a,c);dx), \qquad c\in(a,b),
\end{equation}
where $r$ might or might not be in $L^1((a,c);dx)$. Therefore, if $q\equiv 0$, then Hypothesis \ref{Hypothesis} holds at $x=a$ if and only if one of \eqref{casea} or \eqref{caseb} holds.\\[1mm]
$(iii)$ If $p\equiv r\equiv 1$ in \eqref{SL}, then Hypothesis \ref{Hypothesis} at $x=a$ is equivalent to assuming
\begin{equation*}
\int_a^c u_a^2(x)\int_x^{x_0}\frac{dt}{u_a^2(t)}dx<\infty
\end{equation*}
(cf. \cite[Def. 9.3]{LW23}).
\end{remark}

\section{Properties of the principal solution}
\label{sectphi}
The goal of the present section is to construct a fundamental solution $\varphi(z,x)$ satisfying
\begin{align}\label{tauphi}
    \tau \varphi(z,x) = z \varphi(z,x),
\end{align} 
which is principal at the endpoint $x = a$ and entire in $z \in \mathbb C$ (a similar construction can be performed at the endpoint $x = b$). No additional requirements are needed at $x=b$. The key to this construction is the following technical lemma which we will use repeatedly in the present paper. While we believe it to be known, we were unable to find this result in the literature, so we include the proof in the appendix.
\begin{lemma}\label{LongLem}
    Let $D = U \times (c,d)$, where $a < c < d < b$ and $U \subseteq \mathbb C$ is open. Denote by $T_{max}^{(c,d)}$ the maximal operator associated with $\tau|_{(c,d)}$.
    \begin{enumerate}
        \item Let $y \colon D \to \mathbb C$ be given such that $y(z, \, \cdot \,) \in \emph{dom}(T_{max}^{(c,d)})$ for all $z \in U$. Moreover, assume that $\tau y(z,x) = z y(z,x)$ for $(z,x) \in D$, with $y(z,x)$ being holomorphic in $z$. Then the mapping $z \mapsto y(z, \, \cdot \,)$ is an $L^2((c,d); r(x)dx)$-valued holomorphic mapping and $y$ has locally around $z_0 \in U$ a series expansion
    \begin{align}\label{yPowerSeries}
        y(z,x) = \sum_{n\geq 0} y_n(z_0,x)(z-z_0)^n,
    \end{align}
    where each $y_n$ is in $\emph{dom}(T_{max}^{(c,d)})$  and
    \begin{align}\label{y_n}
        (\tau -z_0) y_0 = 0, \qquad (\tau -z_0)y_{n} = y_{n-1}, \quad n \geq  1.
    \end{align}
    \item Assume that $y \colon D \to \mathbb C$ has locally the series representation \eqref{yPowerSeries} in the space $L^2((c,d); r(x) dx)$ with $y_n \in \emph{dom}(T_{max}^{(c,d)})$ satisfying \eqref{y_n} $($in particular $z \to y(z, \, \cdot \,)$ is an $L^2((c,d); r(x) dx)$-valued analytic mapping$)$. Then $y(z, \, \cdot \, )$ for $z \in U$ is in $\emph{dom}(T_{max}^{(c,d)})$  and satisfies $\tau y(z,x) = z y(z,x)$ for $(z,x) \in D$.
    \end{enumerate}
\end{lemma}
\begin{proof}
See Appendix \ref{AppendixA}.
\end{proof}

We now construct the solution $\varphi(z,x)$ via the infinite power series given by
\begin{align}\label{defPhi}
    \varphi(z,x) = \sum_{n = 0}^\infty \varphi_n(x)(z-\lambda)^n,\qquad x\in(a,b),\ \lambda\in\R.
\end{align}
In fact, rewriting \eqref{tauphi} as
\begin{align*}
    (\tau - z)\varphi = \frac{1}{r}\Big(-(p\varphi')'+ \big(q - \lambda r - (z-\lambda) r\big) \varphi\Big) = 0,
\end{align*}
we see that \eqref{defPhi} is the usual Born series, where we view $(z-\lambda) \in \C$ as the coupling constant for the potential $-r$ (see \cite{GT81} and Remark \ref{RemarkBorn}).

We note that $\varphi_n(x)$ clearly depends on the choice of $\lambda$, so a more precise notation would be $\varphi_n(\lambda,x)$ (see \eqref{yPowerSeries}). However, to keep the notation short we will suppress this $\lambda$-dependence and simply write $\varphi_n(x)$ as is customary with $u$ and $\np$. It will turn out that the choice of $\lambda \in \R$ does not play any significant role (see Cor.~\ref{CorHyp}).

We remark that other spectral parameter power series have been discussed in \cite{KP10}, specifically, the numerical aspects regarding eigenvalue problems (see also the review \cite{KT13a}). An equivalent construction also appeared in \cite{LW23} in relation to the index $\Delta$ mentioned in the Introduction. 

Assuming Hypothesis \ref{Hypothesis} holds at $x=a$, we define $\varphi_0(x) = u_a(x)$, where $u_a$ is a principal solution of \eqref{EigenEq}. That is, we begin by constructing the series \eqref{defPhi} about a point $\lambda\in\R$ such that $\tau f=\lambda f$ is nonoscillatory (though this can be extended to all $\lambda\in\R$ by Corollaries  \ref{CorNorm}--\ref{CorHyp}). We then define iteratively
\begin{align}\label{phin}
    \varphi_n(x) = \int_a^x [ u_a(t)\np_a(x) - \np_a(t)u_a(x)]\varphi_{n-1}(t)r(t) dt,\qquad x\in(a,b), 
\end{align}
where $\np_a$ is a nonprincipal solution of \eqref{EigenEq} satisfying $W(\np_a, u_a) = 1$.
In the following, we fix a $c \in (a,b)$ such that $u_a$, $\np_a$ have no zeros on $(a,c]$. We remark that in this case we also have (see \cite[Thm. 2.2(iii)]{NZ92})
\begin{align}\label{NZIneq}
    |u_a(t)\np_a(x)|<|\np_a(t)u_a(x)|, \qquad a < t < x < c.
\end{align}
In Lemma \ref{LemmaRho} below we prove that the integral \eqref{phin} indeed exists.

As we assume Hypothesis \ref{Hypothesis}, we can define the function
\begin{align}\label{DefRho}
    \rho(x) = \int_a^x |\np_a(t) u_a(t) r(t) |dt,\qquad x\in (a,c).
\end{align}
Note that $\rho(x) \to 0$ for $x \to a^+$. This leads to the following:
\begin{lemma}\label{LemmaRho}
    Assume Hypothesis \ref{HypoInt} and that Hypothesis \ref{Hypothesis} holds at $x=a$. Let $c \in (a,b)$ be chosen such that $u_a, \np_a$ have no zeros on $(a,c]$. Then the following estimates hold for $x \in (a,c)$$:$
    \begin{align}\label{PhiRhoEst}
      |\varphi_n(x)| \leq \rho^n(x)|u_a(x)| , \qquad n \in \mathbb N.  
    \end{align}
\end{lemma}
\begin{proof}
    We proceed by induction. The estimate is trivial for $n = 0$ and let us assume it holds for up to $n-1$. Observe that as by assumption $u_a(x)$, $\np_a(x)$ do not have zeros for $x \in (a,c)$, it follow from \eqref{NZIneq} that
    \begin{align}\label{uuIneq}
        |u_a(t)\np_a(x) - \np_a(t)u_a(x)| < |\np_a(t)u_a(x)|, \qquad a < t < x < c.
    \end{align}
    Using now the monotonicity of $\rho$ together with the induction hypothesis we obtain
    \begin{align*}
        \Big| \int_a^x [ u_a(t)\np_a(x) - \np_a(t)u_a(x)]\varphi_{n-1}(t)r(t) dt& \Big| <  \int_a^x |\np_a(t)u_a(x)\varphi_{n-1}(t)r(t)| dt
        \\
        &< \int_a^x |\np_a(t)u_a(t)r(t)| dt \, \rho^{n-1}(x) |u_a(x)|
        \\
        &= \rho^n(x)|u_a(x)|,
    \end{align*}
    which finishes the proof.
\end{proof}
We now want to show that the series \eqref{defPhi} is indeed entire in $z$. From the previous lemma it follows that \eqref{defPhi} is convergent on $D_{\rho, c} = \lbrace (z,x) : x \in (a,c), \, |z| < \rho^{-1}(x) \rbrace$. Moreover, by definition  $(\tau-\lambda)\varphi_0 = 0$ and a direct calculation shows that
\begin{align*}
    (\tau-\lambda)\varphi_n = \varphi_{n-1}, \qquad n > 0.
\end{align*}
In particular for $(z,x) \in D_{\rho, c}$ it follows by Lemma~\ref{LongLem} $(ii)$ that
\begin{align*}
    \tau \varphi(z,x)  = z\varphi(z,x).
\end{align*}

To show that \eqref{defPhi} converges not only in $D_{\rho, c}$ but in fact defines an entire function in $z$ for all $x \in (a,b)$, let us choose an $x_\varepsilon = a + \varepsilon$, with $\varepsilon > 0$ small enough such that $x_\varepsilon \in (a,c)$. Consider the entire system of solutions $s_\varepsilon(z, x)$ and $c_\varepsilon(z,x)$ of $\tau f = z f$ satisfying 
\begin{equation*}
s_\varepsilon(z,x_\varepsilon) = 0=c_\varepsilon^{[1]}(z,x_\varepsilon), \qquad  s_\varepsilon^{[1]}(z,x_\varepsilon) = 1=c_\varepsilon(z,x_\varepsilon).
\end{equation*}
Let us now define
\begin{align*}
    \overset{\circ}{\varphi}(z,x) = \varphi(z,x_\varepsilon) c_\varepsilon(z,x) +  \varphi^{[1]}(z,x_\varepsilon) s_\varepsilon(z,x), \qquad x \in (x_\varepsilon, b), \quad |z| < \rho^{-1}(x_\varepsilon).
\end{align*}
Note that $\overset{\circ}{\varphi}$ is holomorphic in its first argument and satisfies $\tau \overset{\circ}{\varphi}(z,x) = z\overset{\circ}{\varphi}(z,x)$. By a standard uniqueness results for differential equations we must have $\overset{\circ}{\varphi}(z,x) = \varphi(z,x)$ for $(z,x) \in D_{\rho,c}\cap \dom(\overset{\circ}{\varphi})$. Hence for any fixed $x_0 \in (x_\varepsilon, c)$ it follows that $\varphi(z,x_0)$ can be analytically continued to a holomorphic function in the disc of radius $\rho^{-1}(x_\varepsilon)$ around $\lambda$. Letting $\varepsilon \to 0$ and thus $\rho^{-1}(x_\varepsilon) \to \infty$, we see that $\varphi(z,x)$ is indeed entire in $z$ for all $x\in (a,c)$. To extend this result to $x \in (a,b)$, observe that we can write using Lemma~\ref{LongLem} $(i)$,
\begin{align*}
    \overset{\circ}{\varphi}(z,x) = \sum_{n \geq 0} \overset{\circ}{\varphi}_n(x)(z-\lambda)^n,
\end{align*}
with $\overset{\circ}{\varphi}_n|_{(a,c)} = \varphi_n|_{(a,c)}$ and $(\tau-\lambda) \overset{\circ}{\varphi}_0 = 0$, $(\tau-\lambda) \overset{\circ}{\varphi}_{n} = \overset{\circ}{\varphi}_{n-1}$ for $n > 0$. Again from the uniqueness of solutions to differential equations, we can iteratively conclude that $\varphi_n(x) = \overset{\circ}{\varphi}_n(x)$ for $x \in (a,b)$. We have thus shown the following:
\begin{proposition}\label{PropPhiEntire}
    The infinite series \eqref{defPhi} converges for all $x\in(a,b)$, $z \in \C$, and defines a function $\varphi(z,x)$ which is entire in $z$ and satisfies
    \begin{align*}
        \tau \varphi(z,x) = z \varphi(z,x).
    \end{align*}
\end{proposition}
It should be noted that while the series \eqref{defPhi} converges for all $x\in(a,b)$, the estimate in \eqref{PhiRhoEst} will not hold in general for $x \in (c,b)$. 

We now note a few immediate corollaries from the construction of $\varphi$.
\begin{corollary}\label{CorNorm}
Assume Hypothesis \ref{HypoInt} and that Hypothesis \ref{Hypothesis} holds at $x=a$. Then $\varphi$ defined by \eqref{defPhi}, \eqref{phin} satisfies
\begin{align}\label{phiNorm}
    \lim_{x \to a^+} \frac{\varphi(z_1,x)}{\varphi(z_2,x)} = 1, \qquad z_1, z_2 \in \mathbb C.
\end{align}
In particular, $\tau f = \lambda f$ is nonoscillatory for all $\lambda \in \mathbb R$. 
\end{corollary}
Note that this immediately implies that \eqref{defPhi} can be interpreted as a perturbative Born series, in the sense that higher-order corrections $\varphi_n$, $n \geq 1$ become negligible in the limit $x \to a^+$. This should be contrasted with the case of discrete spectrum but non-trace class resolvents (see Theorem \ref{TFAE} and the subsequent Remark \ref{RemarkBorn}).

The converse of the above corollary will be stated in Corollary \ref{CorConv}. From the previous corollary, we also conclude the following.
\begin{corollary}\label{CorPrinc}
Assume Hypothesis \ref{HypoInt} and that Hypothesis \ref{Hypothesis} holds at $x=a$. Then $\varphi(z, x)$ is principal at $a$ for all $z \in \mathbb R$.
\end{corollary}
\begin{proof}
We know that in the nonoscillatory case at $a$, a solution $f$ of $\tau f = z f$ is principal at $a$ if and only if the function $(pf^2)^{-1}$ is not integrable near the endpoint $x = a$ (see e.g.~\cite[Thm. 2.2(ii)]{NZ92}). Due to \eqref{phiNorm}, $(p\varphi^2(\lambda, \, \cdot \,))^{-1}$ is not integrable near $x = a$ if and only if $(p\varphi^2(z, \, \cdot \,))^{-1}$ is not integrable for any $z \in \mathbb R$.
\end{proof}

Note that as $\varphi(z,x)$ is principal at $x = a$ and $\tau f = zf$ is nonoscillatory at $a$ for all $z\in \R$, we can obtain a nonprincipal solution $\np_a(z, \, \cdot \,)$ of $\tau f = z f$ via the formula $\np_a(z, x) = \varphi(z, x) \int_x^c \frac{dt}{p(t) \varphi^2(z,t)}$, where $c$ is chosen such that $\varphi(z, \, \cdot \, )$ does not vanish on $(a, c]$. In particular, the asymptotic behavior of nonprincipal solutions for $x \to a^+$ is already dictated by the corresponding behavior of the principal solution $\varphi(z,x)$ (cf.~the proof of Lem.~\ref{ThetaLemma}). Thus, \eqref{phiNorm} also implies the independence of Hypothesis \ref{Hypothesis} from the generalized eigenvalue $\lambda \in \R$.
\begin{corollary}\label{CorHyp}
    The Hypothesis \ref{Hypothesis} is independent of $\lambda \in \R$, that is, if it holds for one $\lambda \in \R$  it will hold for all $\lambda \in \R$.
\end{corollary}
As previously pointed out, Corollaries  \ref{CorNorm}--\ref{CorHyp} now imply that the choice $\lambda \in \R$ does not play any significant role in the iterative construction of $\varphi$. We also obtain the following from \cite[Lem.~3.2]{GZ06}.
\begin{corollary}\label{CorGZ}
Assume Hypothesis \ref{HypoInt} and that Hypothesis \ref{Hypothesis} holds at $x=a$. Then all self-adjoint realizations of the restriction $\tau|_{(a,c)}$ to an interval $(a,c)$ with $c \in (a,b)$ have a purely discrete spectrum.
\end{corollary}
\begin{remark}\label{RemarkLaguerre}
    The inverse of Corollary \ref{CorGZ} does not hold. A simple counterexample is given by the Laguerre differential expression, for which all self--adjoint realizations have purely discrete spectrum, but the principal solution at $\infty$ has asymptotically different behavior for different generalized eigenvalues (see \eqref{LaguAsym}), hence Hypothesis \ref{Hypothesis} does not hold. The details can be found in Section \ref{subLag}.
\end{remark}

We next turn to the properties of a second linearly independent fundamental solution $\theta(z,x)$. 

\section{Properties of the nonprincipal solution}
\label{secttheta}
By Corollary \ref{CorGZ} and Remark \ref{RemarkLaguerre} we know that the following hypothesis is weaker than Hypothesis \ref{Hypothesis}.
\begin{hypothesis}\label{HypoGZ}
    Assume that all self--adjoint realizations of the restriction $\tau|_{(a,c)}$ to an interval $(a,c)$ with $c \in (a,b)$ and the Dirichlet boundary condition at $c$ have a purely discrete spectrum. 
\end{hypothesis}
As shown in \cite{GZ06}, Hypothesis \ref{HypoGZ} is equivalent to the existence of an entire fundamental system of solutions $\widetilde \varphi(z,x)$, $\widetilde \theta(z,x)$ of $\tau f = zf$, real on the real axis, such that $\widetilde \varphi(z,x)$ is principal for all $z \in \mathbb R$ and $W(\widetilde \theta(z, \, \cdot \,), \widetilde \varphi(z, \, \cdot \,)) = 1$. The tilde indicates that our standard Hypothesis \ref{Hypothesis} is not assumed, and no additional normalization conditions on $\widetilde \varphi$ and $\widetilde \theta$ are imposed. We now state the following:
\begin{theorem}\label{TFAE}
    Assume Hypotheses \ref{HypoInt}, \ref{HypoGZ} and let $\widetilde \varphi$, $\widetilde \theta$ be chosen as above. Then the following are equivalent:
    \begin{enumerate}
        \item Hypothesis \ref{Hypothesis} at $x=a$ for some $\lambda \in \R;$
        \\
        \item $\displaystyle\lim_{x \to a^+} \cfrac{\widetilde \varphi(z_1,x)}{\widetilde \varphi(z_2,x)} \in \mathbb R \setminus \lbrace 0 \rbrace$ \ for all $z_1, z_2 \in \mathbb R;$
        \\
        \item $\displaystyle\lim_{x \to a^+} W(\widetilde \theta(z_2, x), \widetilde \varphi(z_1, x)) \in \mathbb R \setminus \lbrace 0 \rbrace$ \ for all $z_1, z_2 \in \mathbb R;$
        \\
        \item $\displaystyle\lim_{x \to a^+} \cfrac{\widetilde \theta(z_1,x)}{\widetilde \theta(z_2,x)} \in \mathbb R \setminus \lbrace 0 \rbrace$ \ for all $z_1, z_2 \in \mathbb R;$
        \\
        \item $\int_a^x |\widetilde \theta(z_1,t) \widetilde \varphi(z_2,t) r(t)| dt < \infty$ \ for all $z_1, z_2 \in \mathbb R$ and $x \in (a,b);$
        \\
        \item Hypothesis \ref{hypothesisTrace} at $x = a$.
    \end{enumerate}
    Moreover, in conditions $(ii)$--$(iv)$ `\emph{for all $z_1, z_2 \in \R$}' can be replaced by `\emph{for some distinct $z_1, z_2 \in \R$}'.   
\end{theorem}
\begin{proof}
The equivalence between $(i)$--$(v)$ is rather simple, however point $(vi)$ requires more technical arguments. We provide the complete proof in Appendix \ref{AppendixA}.
\end{proof}

We can now state the converse of Corollary \ref{CorNorm}.
\begin{corollary}\label{CorConv}
    Assume Hypothesis \ref{HypoInt} and let $\widetilde \varphi$ be an entire fundamental solution of $\tau f = zf$ which is principal at $x = a$ for all $z \in \R$. If $\widetilde \varphi$ satisfies
    \begin{align*}
        \lim_{x \to a^+} \frac{\widetilde \varphi(z_1,x)}{\widetilde \varphi(z_2,x)} = 1, \qquad z_1, z_2 \in \mathbb C,
    \end{align*}
    then Hypothesis \ref{Hypothesis} holds at $x = a$ and $\widetilde \varphi$ is equal to $\varphi$ constructed via \eqref{defPhi}, \eqref{phin} up to a multiplicative constant.
\end{corollary}
\begin{remark}\label{RemarkBorn}
    We already observed in Corollary~\ref{CorNorm} that the Born series given through \eqref{defPhi} and \eqref{phin} has the convenient property of the leading term $\varphi_0$ being dominant as $x \to a^+$, that is, higher order terms $\varphi_n$ can be viewed as small corrections. Theorem \ref{TFAE} further tells us that this happens if and only if $\tau|_{(a,c)}$ has self-adjoint realizations with trace class resolvents, meaning that Hypothesis \ref{Hypothesis} is the most general condition under which one can expect a well-behaved Born series with the spectral parameter as the coupling constant. We find it interesting that the condition for the mere existence of an entire fundamental solution $\widetilde \varphi$ which is principal at $x = a$ is significantly weaker, and only requires self-adjoint realizations of $\tau|_{(a,c)}$ to have a purely discrete spectrum (see Hypothesis \ref{HypoGZ}). Being entire, $\widetilde \varphi$ will again have an everywhere convergent power series expansion of the form \eqref{defPhi}, however with $\widetilde \varphi_n$ not necessarily given through \eqref{phin}. As in the absence of the trace class resolvent condition the behavior of $\widetilde \varphi$ must necessarily depend on the spectral parameter $z$ due to Theorem \ref{TFAE} $(ii)$, it follows that higher order terms $\widetilde \varphi_n$ cannot be viewed as small corrections for $x \to a^+$, despite the series being convergent. See Section \ref{subLag} for an explicit example of this phenomenon.  
\end{remark}

Returning to the normalization \eqref{phiNorm}, we obtain the following corollary.

\begin{corollary}\label{CorThetaNorm}
    Assume Hypothesis \ref{HypoInt} holds and let $\varphi$ satisfy the normalization \eqref{phiNorm}. Then any entire fundamental solution $\theta$ satisfying $W(\theta(z, \, \cdot \, ), \varphi(z, \, \cdot \,)) = 1$ will also satisfy
    \begin{align}\label{ThetaWronski}
        \lim_{x \to a^+} \frac{\theta(z_1,x)}{\theta(z_2,x)} = 1, \qquad  \lim_{x \to a^+}  W(\theta(z_2, x), \varphi(z_1, x)) &= 1, \qquad z_1, z_2 \in \mathbb R.
    \end{align}
\end{corollary}
\begin{proof}
    That follows immediately from \eqref{tildePhi} and \eqref{tildeTheta} in the proof of Theorem \ref{TFAE} in Appendix \ref{AppendixA}.
\end{proof}
Note $W(\theta(z, \, \cdot \, ), \varphi(z, \, \cdot \,)) = 1$ implies that $\theta(z, \, \cdot \,)$ is linearly independent of $\varphi(z, \, \cdot \,)$ and hence nonprincipal.

For technical reasons, we will also need that $\lim_{x \to a^+} \frac{\theta(z,x)}{\theta(\lambda,x)}$ converges locally uniformly for $z \in \mathbb C$. This is shown next.
\begin{lemma}\label{ThetaLemma}
    Denote by $h_x(z) = \frac{\theta(z,x)}{\theta(\lambda,x)}$. Then as $x \to a^+$, the entire function $h_x(z)$ converges locally uniformly in $\mathbb C$ to the constant function $1$.
\end{lemma}
\begin{proof} Note that locally in $z$ we can write
\begin{align*}
\theta(z,x) = \varphi(z,x)\int_x^{c_0} \frac{dt}{p(t)\varphi^2(z,t)}+\eta(z)\varphi(z,x),
\end{align*} where $c_0$ is sufficiently close to $a$ such that $\varphi(z,t)$ does not vanish, and $\eta(z)$ is holomorphic. Now as for $x \to a^+$ we have $\varphi(z,x)/\varphi(\lambda,x) = 1 + O(\rho(x))$ with the error being locally uniform in $z$, we conclude that the same is true for $\theta$.
    \end{proof}
    
 As a corollary of Lemma~\ref{ThetaLemma} we can now prove the following (cf.~Lemma \ref{LemmaRho}).
    \begin{corollary}\label{CorTheta}
        Consider the power series expansion of $\theta$ with respect to $z$,
        \begin{align*}
            \theta(z,x) = \sum_{n=0}^\infty \theta_n(x)(z-\lambda)^n.
        \end{align*}
        Then
        \begin{align}\label{ThetaLimit}
            \lim_{x\to a^+}\frac{\theta_n(x)}{\theta_0(x)} = 0 \ \text{ for all} \ n \geq 1.
        \end{align}
    \end{corollary}
    \begin{proof}
        Note that as $h_x(z) = \frac{\theta(z,x)}{\theta(\lambda,x)} \to 1$ converges locally uniformly for $x \to a^+$, we can conclude that $\partial^n_z h_x(z) \to 0$ locally uniformly for $x \to a^+$ and $n \geq 1$. In particular,
        \begin{align*}
        \partial^n_z h_x(z)|_{z=\lambda} = \frac{n!\theta_n(x)}{\theta_0(x)} \to 0 \ \text{ for} \ x \to a^+.
        \end{align*}
    \end{proof}
    Due to Lemma~\ref{LongLem} $(i)$  we know that $(\tau-\lambda)\theta_0(x) = 0$ and $(\tau-\lambda) \theta_n(x) = \theta_{n-1}(x)$ for $n \geq 1$. Hence we can write a general expression for $\theta_n$ in terms of $\theta_{n-1}$
    \begin{align}\nonumber
        \theta_{n}(x) &= A_n \varphi_0(x) + B_n \theta_0(x) 
        \\\label{thetaNformula}
        &\quad + \theta_0(x)\int_a^x \varphi_0(t) \theta_{n-1}(t) r(t) dt + \varphi_0(x) \int_x^c \theta_0(t) \theta_{n-1}(t) r(t) dt.
    \end{align}
    Note that the two integrals above sum up to the usual formula involving the Green's function, however we prefer to keep these integrals separate. 
    
    In the proof of Lemma \ref{Lemma1} we will show that $B_n = 0$, while the constant $A_n$ will be determined later and depends additionally on the choice of $c \in (a,b)$. Note that it is crucial that $\theta_{n-1}(x)= O(\theta_0(x))$ as $x \to a^+$ for all $n \geq 1$, to make sure that the first integral in \eqref{thetaNformula} exists due to Hypothesis \ref{Hypothesis}.

    We finish this section with a few technical results on the behavior of $\theta_n(x)$ as $x \to a^+$, which will be used in the following section.
    \begin{lemma}\label{lemTheta_n}
        The functions $W_n(x) = W(\theta_n(x), \varphi_0(x))$ and $\theta_n(x)$ are nonoscillatory as $x \to a^+$ for all $n \geq 0$. Additionally, the function $\widetilde W_n(x) = W(\theta_n(x), \theta_0(x))$ is nonoscillatory as $x \to a^+$ for $n \geq 1$.
    \end{lemma}
    \begin{proof}
        We proceed by induction. The claim is clearly true for $n= 0$ by Theorem \ref{TFAE} $(iii)$. Let us assume it is true up to $n-1$. Then we can use the general formula 
        \begin{align*}
            \Big[W(f, g)\Big]^{t=\beta}_{t=\alpha} =  \int_\alpha^\beta \big[\big(\tau f(t)\big)g(t) -  f(t) \big(\tau g(t)\big)\big] r(t) dt
        \end{align*}
        to obtain
        \begin{equation*}
            \partial_x W_{n}(x) = \big[\big(\tau \theta_{n}(x)\big)\varphi_0(x) - \big(\tau \varphi_0(x)\big)\theta_{n}(x)\big] r(x)= \theta_{n-1}(x) \varphi_0(x) r(x).
        \end{equation*}
        By the induction hypothesis, the right-hand side is nonoscillatory as $x \to a^+$. Hence, $W_n(x)$ is monotonic as $x \to a^+$, implying that it is nonoscillatory. To see that $\theta_n(x)$ is nonoscillatory, observe that
        \begin{align*}
            \bigg( \frac{\theta_n(x)}{\varphi_0(x)} \bigg)'= -\frac{W_n(x)}{p(x)\varphi^2_0(x)},
        \end{align*}
        where we already know that the right-hand side is nonoscillatory. As before this implies that $\frac{\theta_n(x)}{\varphi_0(x)}$ is monotonic, in particular nonoscillatory near $a$. As $\varphi_0(x)$ is nonoscillatory, the same must be true of $\theta_n(x)$. 

        The proof for $\widetilde W_n$ is similar but no longer requires induction. For $n \geq 1$ we have
        \begin{align*}
            \partial_x \widetilde W_n(x) = \theta_{n-1}(x)\theta_0(x) r(x),
        \end{align*}
        hence $\widetilde W_n$ is monotonic, implying that $\widetilde W_n(x)$ is nonoscillatory, as $x \to a^+$.
    \end{proof}
   Note that we have also shown in the previous proof that $\big(\frac{\theta_n(x)}{\varphi_0(x)}\big)^{\pm 1}$ is monotonic as $x \to a^+$, and we will use this fact later.
   \begin{corollary}\label{CorPhiTheta}
        The function $\big(\frac{\theta_n(x)}{\varphi_0(x)}\big)^{\pm 1}$ is monotonic as $x \to a^+$. In particular the limit $\displaystyle \lim_{x \to a^+} \tfrac{\varphi_0(x)}{\theta_n(x)}$
       exists in $\mathbb R \cup \lbrace \pm \infty \rbrace$.
   \end{corollary}
   Analogously, using that $\widetilde W_n$ is nonoscillatory it follows that $\big(\frac{\theta_n(x)}{\theta_0(x)}\big)^{\pm 1}$ is monotonic as $x \to a^+$ for $n \geq 1$. We need this fact in the proof of Lemma \ref{Lemma1}.
   \begin{corollary}\label{CorTheta2}
       The function $\big(\frac{\theta_n(x)}{\theta_0(x)}\big)^{\pm 1}$ is monotonic as $x \to a^+$ for $n \geq 1$.
   \end{corollary}  

\section{The regularization index}
   \label{Sect_RI}
   
   We now come to the main definition of the present paper. Observe that as $\lim_{x \to a^+} \tfrac{\varphi_0(x)}{\theta_0(x)} = 0$ by the principality of $\varphi$ and the nonprincipality of $\theta$, the regularization index $\ell_a$ is well-defined. 

   \begin{definition}[Regularization index]\label{DefRI}
       Assume Hypothesis \ref{HypoInt} and that Hypothesis \ref{Hypothesis} holds at $x=a$. Let $\varphi$ be given via \eqref{defPhi}, \eqref{phin} and take any entire nonprincipal solution $\theta(z,x)$ satisfying $W(\theta(z, \, \cdot \, ), \varphi(z, \, \cdot \,)) = 1$ which is real for $z \in \R$. Then we define the regularization index $\ell_a \in \mathbb{N}_0 \cup \lbrace \infty \rbrace$ of $\tau$ at the endpoint $x=a$ to be the smallest non-negative integer $\ell_a$ such that
       \begin{align*}
            \lim_{x \to a^+} \frac{\varphi_0(x)}{\theta_n(x)} &= 0 \ \text{ for all } \  n \in \lbrace 0, \dots, \ell_a \rbrace,
    \\
    \lim_{x \to a^+} \frac{\varphi_0(x)}{\theta_{\ell_a+1}(x)} &\not = 0,
       \end{align*}
       in case such an integer exists. Otherwise 
       \begin{align*}
    \lim_{x \to a^+} \frac{\varphi_0(x)}{\theta_n(x)} = 0 \ \text{ for all } \ n \geq 0,
\end{align*}
and we set $\ell_a = \infty$.
\end{definition}

The appropriateness of the terminology `regularization index' will become more apparent in Theorem \ref{Thm:Transformation} and Section \ref{SectWeylm}. For now, we note that, loosely speaking, this index (when $0<\ell_a\leq\infty$) allows one to quantify how far a limit point endpoint is away from being transformed to a limit circle endpoint via Darboux transforms. We will see in Theorem \ref{Thm:N=0} that $\ell_a = 0$ exactly corresponds to the limit circle case, which, if singular, can then be regularized in the sense of Niessen and Zettl (see \cite{NZ92} and \cite[Thm.~8.3.1]{Ze05}). This has implications for the spectral theory of self-adjoint realizations, which will be discussed in Sections \ref{SectSpec} and \ref{SectWeylm}.

It turns out that the regularization index at $a$ (resp., $b$) depends only on $\tau$ as explained
in the following remark.

\begin{remark}\label{remarkexamples}
$(i)$ Note that any other system 
 $\widetilde \theta$, $\widetilde \varphi$ satisfying the assumptions of Definition \ref{DefRI} is related to $\theta$, $\varphi$ via
\begin{align*}
    \widetilde \varphi(z,x) = c \varphi(z,x), \qquad \widetilde \theta(z,x) = c^{-1}\theta(z,x) + f(z) \varphi(z,x),
\end{align*}
where $c \in \R \setminus \lbrace 0 \rbrace$ and $f$ is an entire function that is real on the real axis. It is easy to see that the choice $\widetilde \theta$, $\widetilde \varphi$ lead to the same regularization index, so $\widetilde \ell_a = \ell_a$. The question whether the regularization index depends on the choice of $\lambda \in \R$ is the content of Corollary \ref{Cor:RI}. It turns out that $\ell_a$ is independent of $\lambda$, meaning that the regularization index at $a$ $($resp., $b$$)$ depends only on the Sturm--Liouville differential expression $\tau$.\\[1mm]
$(ii)$ Returning to Remark \ref{remarkcoefficients}, we provide a few general examples demonstrating how the regularization index depends on the behavior of $p, r, q$ near singular endpoints. Here and in the following the notation $f(x) \propto g(x)$ (for $x \to a^+$) will be shorthand for $\limsup_{x\to a^+} \big| \frac{f(x)}{g(x)} \big|^{\pm 1} \in \R_+$.

$(a)$ Let $a\in\R$. Assume $q\equiv 0$, $p(x)\propto (x-a)^\nu$, and $r(x)\propto (x-a)^\d$ with $\d,\nu\in\R$. By Remark \ref{remarkcoefficients} $(ii)$, we must restrict the powers to satisfy the integrability conditions given there. Independently of the integrability of $1/p$, this condition reads $2-\nu+\d>0$. If $\nu \geq 1$, that is $1/p \notin L^1((a,c); dx)$, and $2+\d>\nu\geq 1$, then $\varphi_0(x)$ is constant and $\theta_n(x)\propto (x-a)^{n(2-\nu+\d)+1-\nu}$ (times a possible logarithm) for $\nu\geq1$. Thus one concludes $\ell_a=\big\lfloor\frac{\nu-1}{2-\nu+\d}\big\rfloor$ in this case. If $\nu< 1$, that is $1/p \in L^1((a,c); dx)$, and $\nu-2<\d$, then $\varphi_0(x)\propto (x-a)^{1-\nu}$ and $\theta_n(x)\propto (x-a)^{n(2-\nu+\d)}$ (times a possible logarithm)  so that $\ell_a=\big\lfloor\frac{1-\nu}{2-\nu+\d}\big\rfloor$ in this case. Combining these two cases gives the regularization index as $\ell_a=\big\lfloor\frac{|\nu-1|}{2-\nu+\d}\big\rfloor$ in general. The logarithms show up if $n \geq \ell_a$ and either $\ell_a = \frac{\nu-1}{2-\nu+\d}$ for $\nu \geq 1$, or $\ell_a = \frac{1-\nu}{2-\nu+\d}$ for $\nu < 1$. Notice that the index is always finite and $\theta_n\in L^2((a,c);r(x)dx)$ for $n>(\ell_a-1)/2$.

$(b)$ Let $a\in\R$. If one assumes $p\equiv r\equiv 1$ and $\varphi_0(x) \propto (x-a)^\alpha$ with $\alpha\geq1/2$, then Hypothesis \ref{Hypothesis} holds by Remark \ref{remarkcoefficients} $(iii)$ (note $\alpha \geq 1/2$ is necessary as otherwise $\varphi_0^{-2}$ would be integrable near $a$ contradicting its principality). Furthermore, $x=a$ is limit circle for $\alpha\in[1/2,3/2)$ and $\theta_n(x)\propto (x-a)^{1-\alpha+2n}$ (times a possible logarithm) for $\alpha \geq1/2$ so that $\ell_a = \big\lfloor\alpha-\frac{1}{2}\big\rfloor$. Once again, the index is finite and $\theta_n\in L^2((a,c);dx)$ for $n>(\ell_a-1)/2$.
\end{remark}

The following lemma is one of the main structural results of the present paper and will enable us to interpret the regularization index as an extension of the binary limit circle/limit point classification.
\begin{lemma}\label{Lemma1}
Assume Hypothesis \ref{HypoInt} and that Hypothesis \ref{Hypothesis} holds at $x=a$. Then 
\begin{align}\label{ThetaRatio}
    \lim_{x \to a^+} \frac{\theta_{n+1}(x)}{\theta_{n}(x)} = 0 \ \text{ for all } \  n \in \lbrace 0, \dots, \ell_a \rbrace.
\end{align}
\begin{proof}
    We will use the representation \eqref{thetaNformula}. Let us assume that $n \leq \ell_a$.  First, observe that
    \begin{align*}
\frac{\theta_0(x)\int_a^x \varphi_0(t) \theta_{n}(t) r(t) dt}{\theta_n(x)}  &= \int_a^x \big(\varphi_0(t) \theta_{0}(t) r(t)\big)\frac{\frac{\theta_n(t)}{\theta_0(t)}}{\frac{\theta_n(x)}{\theta_0(x)}} dt
\\
&=  \int_a^x \big(\varphi_0(t) \theta_{0}(t) r(t)\big)\frac{F_n(t)}{F_n(x)} dt.
    \end{align*}
    Where we defined $F_n(s) = \frac{\theta_n(s)}{\theta_0(s)}$. We know from Corollary \ref{CorTheta} and Corollary \ref{CorTheta2} that $F_n(s) \to 0$ monotonically as $s \to a$. In particular, $\Big|\frac{F_n(t)}{F_n(x)}\Big| \leq 1$ for $t \in (a,x)$ and $x$ sufficiently close to $a$. Thus it follows that
    \begin{align*}
        \lim_{x \to a^+} \frac{\theta_0(x)\int_a^x \varphi_0(t) \theta_{n}(t) r(t) dt}{\theta_n(x)} = 0.
    \end{align*}
    Next, let us consider
    \begin{align*}
        \frac{\varphi_0(x) \int_x^c \theta_0(t) \theta_{n}(t) r(t) dt}{\theta_n(x)} &= \int_x^c \big(\varphi_0(t) \theta_{0}(t) r(t)\big)\frac{\frac{\varphi_0(x)}{\theta_n(x)}}{\frac{\varphi_0(t)}{\theta_n(t)}} dt
\\
&=  \int_x^c \big(\varphi_0(t) \theta_{0}(t) r(t)\big)\frac{G_n(x)}{G_n(t)} dt,
    \end{align*}
    where $G_n(s) = \frac{\varphi_0(s)}{\theta_n(s)}$. Now observe that as $n \leq \ell_a$, we have by Definition \ref{DefRI} that $\lim_{s\to a} G_n(s) = 0$, and this convergence is monotonic due to Corollary \ref{CorPhiTheta}. In particular, we can assume that $c$ is chosen close enough to $a$ such that $\Big|\frac{G_n(x)}{G_n(t)} \Big| \leq 1$ for $a < x \leq t < c$. An application of dominated convergence gives us
    \begin{align*}
        \lim_{x \to a^+} \frac{\varphi_0(x) \int_x^c \theta_0(t) \theta_{n}(t) r(t) dt}{\theta_n(x)} = 0.
    \end{align*}
    It now follows from \eqref{ThetaLimit} that necessarily $B_n = 0$. Moreover, as $n \leq \ell_a$ we have $\lim_{x\to a^+} \frac{A_n \varphi_0(x)}{\theta_n(x)} = 0$. Hence, \eqref{ThetaRatio} follows, finishing the proof. 
\end{proof}
\end{lemma}
Note that we still have some freedom in choosing the $A_n$, which is expected as $\theta(z,x)$ is only unique up to additions of entire multiples of $\varphi(z,x)$. We will see in Section \ref{sectnatural} that further conditions need to be imposed on $\theta$ through the choice of $A_n$ to guarantee that \eqref{ThetaRatio} holds for all $n \geq 0$ (see Theorem \ref{Theorem2}).

We continue this section with the following result on the stability of the regularization index $\ell_a$ under perturbations of the potential $q$. The proof is modeled on Lemma 2.2 in \cite{KST_Inv}.
\begin{theorem}\label{ThmPer}
    Assume Hypothesis \ref{HypoInt} and that Hypothesis \ref{Hypothesis} holds at $x=a$. Choose $u_a$, $\np_a$ to be principal resp.~nonprincipal solutions of $\tau f = \lambda f$, $\lambda \in \R$ satisfying $W(\np_a, u_a) = 1$ and define the perturbed Sturm--Liouville differential expression 
    \begin{align*}
\tau^{per} = \dfrac{1}{r(x)}\Big[-\dfrac{d}{dx} p(x)  \dfrac{d}{dx} + q^{per}(x) \Big] \ \text{ for a.e.~$x\in(a,b) \subseteq \R$,}
    \end{align*}
    where the potential $q^{per} = q + q_0$ satisfies  $q_0 \in L^1_{loc}((a,b); dx)$ and
    \begin{align}\label{q0}
        \int_a^{b} |u_a(x) \np_a(x) q_0(x)| dx < \infty.
    \end{align}
    Then $\tau^{per}$ will satisfy Hypothesis \ref{Hypothesis} and the regularization indices of $\tau$ and $\tau^{per}$ at $x = a$ coincide, that is, $\ell_a^{per} = \ell_a$.
\end{theorem}
\begin{proof}
We will use a similar Born series construction as for $\varphi(z,x)$. Let $u^{per}_{a,0}(x) = u_{a}(x)$ and define iteratively using the Green's function
    \begin{align*}
        u^{per}_{a,n}(x) =\int_a^x \big[ u_a(t)\np_a(x) - \np_a(t)u_a(x)\big] \big(-q_0(t)u^{per}_{a,n-1}(t) \big) dt,
        \\
        x\in(a,b), \quad n \geq 1. 
    \end{align*}
    Note that we then have $(\tau-\lambda) u_{a,n}^{per} = \frac{-q_0}{r} u_{a,n-1}^{per}$ for $n \geq 1$ and $(\tau - \lambda)u_{a,0}^{per} = 0$.
     Let $c \in (a,b)$ be chosen such that $u_a(x), \np_a(x)$ do not have any zeros for $x \in (a,c)$. Then just as in the proof of Lemma \ref{LemmaRho} we can show inductively that $|u_{a,n}^{per}(x)| \leq \sigma^n(x) |u_a(x)|$ for $x \in (a,c)$ and $\sigma(x) = \int_a^x |u_a(t) \np_a(t) q_0(t)| dt$. Note that  we have $\sigma(x) \to 0$ as $x \to a^+$. It follows that $u_a^{per}(x) = \sum_{n=0}^\infty u_{a, n}^{per}(x)$ converges for $x$ close enough to $a$, and $u_a^{per}(x) = u_a(x)[1+O(\sigma(x))]$. One can check that $\tau^{per} u_a^{per} = \lambda u_a^{per}$ and that $1/[p(u_a^{per})^2]$ is not integrable, meaning that $u_a^{per}$ is a principal solution for $\tau^{per}f = \lambda f$. As the regularization index $\ell_a^{per}$ only depends on the behavior of the principal (or nonprincipal) solution as $x \to a^+$ if the coefficients $p, r$ are fixed, it follows that $\ell_a^{per} = \ell_a$. 
\end{proof}
Provided Hypothesis \ref{Hypothesis} is satisfied, we can always chose $q_0 = \widetilde \lambda r$ and \eqref{q0} will hold. Note that this just corresponds to a spectral shift of $\lambda$. Thus we have shown
\begin{corollary} \label{Cor:RI}
    The regularization index $\ell_a$ is independent of the choice of $\lambda \in \R$ in \eqref{defPhi}.
\end{corollary}
\begin{remark}\label{RemarkPer}
The perturbation condition \eqref{q0} generalizes the condition in \cite[Hypo.~2.1]{KST_Inv} for perturbed spherical Schr\"odinger operators. In fact, the unperturbed case treated in \cite{KST_Inv} corresponds to 
\begin{align*}
    \tau = H_l = -\frac{d^2}{dx^2} + \frac{l(l+1)}{x^2}, \qquad x \in (0,1), \quad l \geq -1/2.
\end{align*}
Entire principal and nonprincipal solutions are given in terms of Bessel functions and satisfy
\begin{align*}
    \varphi(z,x) &\propto x^{l + 1}, \quad x \to 0^+,
    \\
    \theta(z,x) &\propto \begin{cases}
        x^{-l}, & l > -1/2
        \\
        x^{1/2} \ln(x), & l = -1/2
    \end{cases}, \quad \ x \to 0^+.
\end{align*}
Thus for $\varphi \theta q_0$ to be integrable we need to require $q_0 \in L^1_{loc}((0,1);dx)$ and
\begin{align*}
    \begin{cases}
        x q_0(x) \in L^1((0,1);dx), & l > -1/2,
        \\
        x\ln(x) q_0(x) \in L^1((0,1);dx), & l =-1/2. 
    \end{cases}
\end{align*}
This is equivalent to Hypothesis 2.1 stated in \cite{KST_Inv} and allows for the inclusion of the classical Coulomb case $q_0(x)=C/x$.
\end{remark}

\section{Relation to limit circle/limit point classification}\label{sectlclp}

We will now describe the relationship between the regularization index $\ell_a$ and the limit circle/limit point classification of $\tau$ at $a$.
\begin{theorem}\label{Thm:N=0}
    Assume Hypothesis \ref{HypoInt} and that Hypothesis \ref{Hypothesis} holds at $x=a$. Then $\tau$ is in the limit circle case at $x = a$ if and only if $\ell_a = 0$.
\end{theorem}
\begin{proof}
    First, let us assume that $\tau$ is in the limit circle case at $x = a$ (recall that Hypothesis \ref{Hypothesis} always holds in the limit circle case). We can then write (see \eqref{thetaNformula})
    \begin{align}\label{theta1}
        \theta_{1}(x) &= A_1 \varphi_0(x)  +\theta_0(x)\int_a^x \varphi_0(t) \theta_{0}(t) r(t) dt + \varphi_0(x) \int_x^c \theta_0^2(t) r(t) dt
        \\\nonumber
        &= A_1'\varphi_0(x) + \int_a^x \big[\theta_0(x)\varphi_0(t)-\varphi_0(x)\theta_0(t) \big]\theta_0(t) r(t) dt,
    \end{align}
    with $A_1' = A_1 + \int_a^c \theta_0^2(t) r(t) dt$, which exists due to the limit circle assumption. Using \eqref{uuIneq} with $\varphi_0 = u_a$ and $\theta_0 = \np_a$, we conclude that for $x$ close enough to $a$,
    \begin{align*}
        |\theta_0(x)\varphi_0(t)-\varphi_0(x)\theta_0(t)| < |\varphi_0(x)\theta_0(t)|, \qquad a < t < x
    \end{align*}
    holds. In particular for $x \to a^+$ we have 
    \begin{align*}
        \Bigg |\frac{\int_a^x \big[\theta_0(x)\varphi_0(t)-\varphi_0(x)\theta_0(t) \big]\theta_0(t) r(t) dt}{\varphi_0(x)} \Bigg| \leq \Big| \int_a^x \theta_0^2(t) r(t) dt \Big| \to 0,
    \end{align*}
    implying $\lim_{x \to a^+} \frac{\theta_1(x)}{\varphi_0(x)} = A_1'$, showing that indeed $\ell_a = 0$.

    Now assume that $\ell_a = 0$. We will again use the above integral representation of $\theta_1$, that is, the first line of \eqref{theta1}, with $c$ chosen such that $\varphi_0(t)$, $\theta_0(t)$ have no zeros on $(a,c)$. As $\ell_a = 0$, we have that the limit
    \begin{align}\label{LimitR}
        \lim_{x \to a^+} \frac{\theta_0(x)\int_a^x \varphi_0(t) \theta_{0}(t) r(t) dt + \varphi_0(x) \int_x^c \theta_0^2(t) r(t) dt}{\varphi_0(x)} \in \mathbb R,
    \end{align}
    exists. Note that due to our assumption on $c$, both terms in the numerator have the same sign (namely the sign of $\varphi_0$ close to $a$). In particular, this implies
    \begin{align*}
        \limsup_{x \to a^+}\Bigg|\frac{\varphi_0(x) \int_x^c \theta_0^2(t) r(t) dt}{\varphi_0(x)}\Bigg| = \limsup_{x \to a^+}\Big| \int_x^c \theta_0^2(t) r(t) dt \Big| < \infty,
    \end{align*}
    (due to monotonicity one can substitute $\lim$ for $\limsup$) which shows that $\tau$ is in the limit circle case at $x = a$.
\end{proof}

As a corollary, we have the following.

\begin{corollary}
    Assume Hypothesis \ref{HypoInt}, that Hypothesis \ref{Hypothesis} holds at $x=a$ and theta $\theta$ satisfies \eqref{ThetaWronski}. Then $\tau$ is in the limit point case at $x=a$ if and only if
    \begin{align*}
        \lim_{x\to a^+} \frac{\varphi_0(x)}{\partial_z\theta(z,x)|_{z=0}} = c\in\R.
    \end{align*}
\end{corollary}

This brings up the following general open problem which can now be understood as an extension of studying stability of the limit circle/limit point classification:

\begin{problem}
Under which types of perturbation of the coefficient functions is a finite regularization index of $\tau$ at $x=a$ stable?
\end{problem}
Note that Theorem \ref{ThmPer} gives one class of perturbations under which the regularization index remains stable.

For general $\ell_a < \infty$ we also have the following characterization.

\begin{proposition}
    Assume Hypothesis \ref{HypoInt}, that Hypothesis \ref{Hypothesis} holds at $x=a$,  and $\ell_a < \infty$. Then \begin{align}\nonumber
    \theta_0\theta_n &\not \in L^1((a,c); r(x)dx), \qquad n \in \lbrace 0, \dots, \ell_a -1 \rbrace,
    \\\label{IntThetaNa}
    \theta_0\theta_{\ell_a} & \in L^1((a,c); r(x)dx).
\end{align} 
In case $\ell_a = \infty$, we have $ \theta_0\theta_n \not \in L^1((a,c); r(x)dx)$ for all $n \geq 0$.
\end{proposition}
\begin{proof}
    The proof is a simple adaptation of the second part of the proof of Theorem \ref{Thm:N=0}. Let us first show that $\theta_0\theta_{\ell_a}  \in L^1((a,c); r(x)dx)$. By definition we have that (cf.~\eqref{LimitR})
    \begin{align}\label{supThetan}
        \lim_{x \to a^+} \frac{\theta_0(x)\int_a^x \varphi_0(t) \theta_{\ell_a}(t) r(t) dt + \varphi_0(x) \int_x^c \theta_0(t)\theta_{\ell_a}(t) r(t) dt}{\varphi_0(x)} \in \mathbb R.
    \end{align}
    As $\theta_{n}$ is nonoscillatory for $x \to a^+$ by Lemma~\ref{lemTheta_n}, we can assume that $c$ is chosen such that $\varphi_0$, $\theta_0$, and $\theta_{\ell_a}$ have no zeros on $(a,c)$.  Thus, both terms in the numerator of \eqref{supThetan} have the same sign, so we can conclude that
    \begin{align*}
        \limsup_{x \to a^+}\Bigg|\frac{\varphi_0(x) \int_x^c \theta_0(t)\theta_{\ell_a}(t) r(t) dt}{\varphi_0(x)}\Bigg| = \limsup_{x \to a^+}\Big| \int_x^c \theta_0(t)\theta_{\ell_a}(t) r(t) dt \Big| < \infty.
    \end{align*}
    Again, as $\theta_n$ is nonoscillatory at $x = a$, it follows that \begin{align*}
        \limsup_{x \to a^+}\int_x^c |\theta_0(t)\theta_{\ell_a}(t)| r(t) dt  < \infty,
    \end{align*} 
    showing that  $\theta_0\theta_{\ell_a}  \in L^1((a,c); r(x)dx)$.

    Now assume that  $\theta_0\theta_{n}  \in L^1((a,c); r(x)dx)$. Then we can define similarly to \eqref{phin}
    \begin{align*}
         \Theta_{n+1}(x) = \int_a^x [\varphi_0(t)\theta_0(x)-\theta_0(t)\varphi_0(x)]\theta_n(t)r(t)dt, \qquad x \in (a,b),
    \end{align*}
    as the integral converges. Note that $\Theta_{n+1}$ differs from $\theta_{n+1}$ by at most a multiple of $\varphi_0$. From \eqref{NZIneq} we thus have for $x$ close enough to $a$ that
    \begin{align*}
        |\Theta_{n+1}(x)| \leq |\varphi_0(x)| \int_a^x |\theta_0(t) \theta_n(t) r(t)| dt, 
    \end{align*}
    implying that $\lim_{x\to a^+}\frac{\Theta_{n+1}(x)}{\varphi_0(x)} = 0$. Hence we must have $n \geq \ell_a$. This argument also shows the statement regarding $\ell_a = \infty$, thus finishing the proof.
\end{proof}
Note that the above proposition implies that $\theta_{\ell_a} \in L^2((a,c); r(x)dx)$ due to \eqref{ThetaLimit}. However, it can happen that $\theta_n \in  L^2((a,c); r(x)dx)$ for $n < \ell_a<\infty$. For example, the generalized Bessel operator in Section \ref{subBessel} and the Jacobi equation in Section \ref{subJac} both show that this is satisfied for roughly half of $n<\ell_a$. This leads to the following open problem:

\begin{problem}\label{problemTheta}
If $\ell_a<\infty$, can one characterize for what $n<\ell_a$ one has $\theta_n(x) \in  L^2((a,c); r(x)dx)$?
\end{problem} 

Open Problem \ref{problemTheta} is closely related to the question of equivalence between $\ell_a$ and the index $\Delta$ studied \cite{LW23}, as explained in the following remark.
\begin{remark}\label{RemarkDelta}
In \cite{LW23} an index $\Delta$ related to canonical systems is considered, which is defined through an $L^2$-condition. In the Schr\"odinger case $p \equiv r \equiv 1$ the is index is defined through the following procedure (see \cite[Eq.~(9.3)]{LW23}). Choose a principal solution $\varphi_0$ of $\tau f = 0$ and define recursively the following sequence:
\begin{align*}
    \widetilde w_0(x) &= \frac{1}{\varphi_0(x)},
    \\
    \widetilde w_k(x) &= \begin{dcases}
        \varphi_0(x) \int_x^c\frac{1}{\varphi_0(t)} \widetilde w_{k-1}(t) dt, \qquad \text{if } k \text{ is odd},
        \\
        \\
        \frac{1}{\varphi_0(x)}\int_a^x \varphi_0(t)\widetilde w_{k-1}(t) dt, \qquad \text{if } k \text{ is even}.
    \end{dcases}
\end{align*}
It turns out that $\tau \widetilde w_{k+2} =  \widetilde w_k$ for all odd $k$. In fact, the above recursion is equivalent the recursion \eqref{thetaNformula} with $B_n = 0$ written in two steps, and $\widetilde w_1$ is a nonprincipal solution of $\tau f = 0$. This implies that one can identify $\theta_n = \widetilde w_{2n+1}$.

The authors then define the index $\Delta_{Schr}$ to be the smallest $k$ such that $\widetilde w_k \in L^2((a,c);dx)$. Thus, this index is closely related to local $L^2$-integrability of the $\theta_n$. A similar condition is used in the case $q \equiv 0$ and either $1/p$ or $r$ not in $L^1((a,c);dx)$, leading to the indices $\Delta_{SL}$, respectively, $\Delta_{SL}^+$. Comparing \cite[Prop.~7.18, Ex.~9.5]{LW23} to our Remark \ref{remarkexamples} it would appear that in these cases one would have $\ell_a + 1 = \Delta_{SL}$, $\ell_a = \Delta^+_{SL}$ and $\ell_a + 1 = \Delta_{Schr}$. 
\end{remark}

Based on these considerations, we expect that the correct answer to Open Problem \ref{problemTheta} would be $\theta_n \in L^2((a,c); r(x)dx)$ if and only if $n \geq \kappa = \floor{\frac{\ell_a+1}{2}}$ (cf.~Cor.~\ref{CorKappa}). This can be directly checked in the setting of Remark \ref{remarkexamples} $(ii)$ where $q \equiv  0$ and $p, r$ have power like behavior at $x = a$.

\section{natural normalization}
\label{sectnatural}

Note that so far there is plenty of freedom in choosing the nonprincipal solution $\theta$. In fact, $\widetilde \theta(z,x) = \theta(z,x) + f(z)\varphi(z,x)$, with $f$ being entire and real-valued on $\mathbb R$, will still satisfy all the results from the previous section. We now introduce another stronger normalization requirement by additionally including  condition $(iii)$ below, which narrows down the class of admissible $\theta$'s. It turns out that this condition is very convenient in the study of Darboux transforms in Section \ref{SectDarb} and guarantees that the corresponding Weyl $m$-function lies in a suitable generalized Nevanlinna class $N_\kappa^\infty$ in Section \ref{SectWeylm}. 

\begin{definition}[Naturally normalized system]\label{DefThetaNorm}
    Assume Hypothesis \ref{HypoInt} and that Hypothesis \ref{Hypothesis} holds at $x=a$. We call the system of entire solutions $\varphi$, $\theta$ of $\tau f = zf$ \emph{naturally normalized} $($at $x=a$$)$ if and only if
    \begin{enumerate}
        \item $\varphi(z, \, \cdot \,)$ is principal at $x = a$ for all $z \in \mathbb \R$ and $\lim_{x\to a^+}\dfrac{\varphi(z_1,x)}{\varphi(z_2,x)} = 1$ for $z_1, z_2 \in \R;$
        \\
        \item $W(\theta(z, \, \cdot \,), \varphi(z,\, \cdot \,)) = 1$ for all $z \in \R;$
        \\
        \item $\lim_{x\to a^+}\dfrac{\theta_n(x)}{\varphi_0(x)} = 0$ for all $n > \ell_a$.
    \end{enumerate}
    As $\varphi$ is fixed up to multiplicative constants by $(i)$, we will sometimes also say that $\theta$ is \emph{naturally normalized} if it satisfies $(ii)$ and $(iii)$, where $(i)$ is implicitly assumed.
\end{definition}
    Note that by Corollary \ref{CorThetaNorm}, if $\varphi$, $\theta$ are naturally normalized then we must also have $\lim_{x\to a^+}\frac{\theta(z_1,x)}{\theta(z_2,x)} = 1$ for $z_1, z_2 \in \R$. Also, if $\ell_a = \infty$ condition $(iii)$ is vacuous. A more complete description of naturally normalized systems $\varphi$, $\theta$ is given in Theorem~\ref{Theorem2}. Moreover, in the special case of the perturbed spherical Schr\"odinger operator, condition $(iii)$ defines a `Frobenius solution' in the sense of \cite[Def.~3.10]{KT_JDE}. 

\begin{remark}
We remark that it would be of interest to directly compare our notion of a naturally normalized system of solutions to those used in \cite{LW23}. In particular, it is shown there that in the presence of a finite index every solution of the Sturm--Liouville problem (more generally, canonical systems) attains regularized boundary values in the sense that finitely many divergent terms are discarded in defining the values (see \cite[Thms. 4.2, 7.4, and 9.6]{LW23}). The regularized boundary values are then used to fix a fundamental system of solutions in order to construct a singular Weyl $m$-function, much as we do in Section \ref{SectWeylm} via a naturally normalized system.
\end{remark}

    It will follow from the proof of Lemma \ref{LemThetaBound} that given $\varphi$ defined via \eqref{defPhi}, \eqref{phin} (which is unique up to multiplicative constants), the natural normalization condition is equivalent to the recursion
\begin{align}\label{normTheta}
        \theta_{n}(x) = \int_a^x \big[\theta_0(x)\varphi_0(t)-\varphi_0(x)\theta_0(t) \big]\theta_{n-1}(t) r(t) dt \ \text{ for all} \ n > \ell_a.
    \end{align}
    In fact, \eqref{normTheta} can also be used to simply define $\theta_n$ for all $n > \ell_a$ from some initial choice of $\theta_k$ with $k = 0, \dots, \ell_a$ coming from the recursion \eqref{thetaNformula}, implying that at least one naturally normalized $\theta$ always exists. For uniqueness, see Remark \ref{RemThetaUnique}.

\begin{lemma}\label{LemThetaBound}
    Assume Hypothesis \ref{HypoInt}, that Hypothesis \ref{Hypothesis} holds at $x=a$,  and let $\theta_n$ for $n > \ell_a$ be defined inductively by \eqref{normTheta}. Then there exists $c \in (a,b)$ such that $\varphi_0(x)$, $\theta_0(x)$ have no zeros for $x\in (a,c]$ and 
\begin{align}\label{thetanbound}
    |\theta_n(x)| \leq \rho^{n-\ell_a-1}(x)|\varphi_0(x)| \ \text{ for } x \in (a,c) \ \text{ and } \ n > \ell_a
\end{align}
holds, where $\rho$ is defined in \eqref{DefRho}. In particular, all integrals \eqref{normTheta} exist.
\end{lemma}
\begin{proof}
    The proof is similar to the proof of Lemma~\ref{LemmaRho}. We will show the statement via induction. For $n = \ell_a + 1$ and choosing $c$ close enough to $a$ such that\eqref{NZIneq} holds we obtain
    \begin{align}\nonumber
        |\theta_{\ell_a+1}(x)| &= \Big|\int_a^x \big[\theta_0(x)\varphi_0(t)-\varphi_0(x)\theta_0(t) \big]\theta_{\ell_a}(t) r(t) dt\Big|
        \\\label{ThetaN+1Est}
        &\leq \int_a^x |\varphi_0(x)\theta_0(t) \theta_{\ell_a}(t) r(t) |dt = |\varphi_0(x)| \underbrace{\int_a^x |\theta_0(t) \theta_{\ell_a}(t) r(t) |dt}_{\to 0 \text{ by } \eqref{IntThetaNa}}.
    \end{align}
    If necessary, redefining $c$ such that the last integral above is $\leq 1$ it follows that $|\theta_{\ell_a+1}(x)| \leq |\varphi_0(x)|$ for $x\in (a,c]$, showing \eqref{thetanbound} for $n = \ell_a+1$. Now using induction and monotonicity of $\rho$, we obtain for $n > \ell_a + 1$
    \begin{align*}
        |\theta_{n}(x)| &= \Big|\int_a^x \big[\theta_0(x)\varphi_0(t)-\varphi_0(x)\theta_0(t) \big]\theta_{n-1}(t) r(t) dt\Big|
        \\
        &\leq \int_a^x |\varphi_0(x)\theta_0(t) \theta_{n-1}(t) r(t) |dt
        \\
         &\leq \int_a^x |\theta_0(t) \varphi_0(t)  r(t) |dt \, \rho^{n-\ell_a-2}(x)|\varphi_0(x)|
         \\
         &=  \rho^{n-\ell_a-1}(x)|\varphi_0(x)|,
    \end{align*}
    finishing the proof.
\end{proof}
One can now proceed as in Proposition~\ref{PropPhiEntire} to show that $\theta(z,x) = \sum_{n=0}^\infty \theta_n(x)(z-\lambda)^n$ will define an entire function for each $x \in (a,b)$ and $\lambda\in \R$. 

\begin{remark}\label{RemThetaUnique}
Note that if $\theta_I$ and $\theta_{II}$ are both naturally normalized, one has \begin{align*}
    \theta_{II}(z,x) = \theta_I(z,x) + f(z)\varphi(z,x),
\end{align*} where $f$ is a real polynomial of degree $\ell_a$. If $f$ were of a higher degree or a non-polynomial entire function, \eqref{thetanbound} could not be satisfied by both $\theta_I$ and $\theta_{II}$. The appearance of the polynomial $f$ comes from the freedom to choose $\ell_a$ arbitrary real constants $A_1, \dots, A_{\ell_a}$ in \eqref{thetaNformula}, together with the fact that $\theta_0$ is also fixed only up to the addition of a constant multiple of $\varphi_0$. In particular,  in the limit circle case,~$\ell_a = 0$, our normalization condition fixes $\theta$ up to the addition of a constant multiple of $\varphi$. This is not surprising, as we just recover the usual normalization for nonoscillatory singular Sturm--Liouville operators in the limit circle case, that is,~$\varphi$ and $\theta$ will be normalized in the sense of \cite{GLN20}. 
\end{remark}

Next, we want to refine Lemma \ref{Lemma1} in the case when the system $\varphi$, $\theta$ is naturally normalized. For this, we will make use of the following lemma.
\begin{lemma}\label{lemmaFF}
    Assume $f_1$, $f_2$ are nonoscillatory near $x=a$, such that 
    \begin{align*}
        F_j(x) = \int_a^x \big[\theta_0(x)\varphi_0(t)-\varphi_0(x)\theta_0(t) \big]f_j(t) r(t) dt, \qquad x \in (a,b), \quad j = 1,2,
    \end{align*}
    exist, that is,~the integrals converge. Then $F_j$ are nonoscillatory and if $\displaystyle\lim_{x \to a^+} \tfrac{f_1(x)}{f_2(x)} = 0$, then $\displaystyle\lim_{x \to a^+} \tfrac{F_1(x)}{F_2(x)} = 0$.
\end{lemma}
\begin{proof}
    Recall that $\theta_0(x)\varphi_0(t)-\varphi_0(x)\theta_0(t)$ with $t \in (a,x)$ will have no sign changes for $x$ close enough to $a$ (see \eqref{NZIneq}). The same is true for $f_j$ by assumption. From this, it follows that $F_j$ are monotonic, hence nonoscillatory. Thus we can estimate
    \begin{align*}
        \lim_{x \to a^+} \Big|\frac{F_1(x)}{F_2(x)}\Big| &= \lim_{x \to a^+} \Bigg|\frac{\int_a^x \big[\theta_0(x)\varphi_0(t)-\varphi_0(x)\theta_0(t) \big]f_1(t) r(t) dt}{\int_a^x \big[\theta_0(x)\varphi_0(t)-\varphi_0(x)\theta_0(t) \big]f_2(t) r(t) dt} \Bigg|
        \\
        &\leq \lim_{x \to a^+} \Bigg(\sup_{t\in(a,x]}\Big| \frac{f_1(t)}{f_2(t)} \Big|\Bigg) = 0,
    \end{align*}
    finishing the proof.
\end{proof}
We state now the refinement of Lemma \ref{Lemma1} for the case that $\theta$ is normalized according to Definition \ref{DefThetaNorm}.
\begin{theorem}\label{Theorem2}
    Assume Hypothesis \ref{HypoInt}, that Hypothesis \ref{Hypothesis} holds at $x=a$, and let $\varphi$, $\theta$ be naturally normalized. Then for all $k \geq 0$ we have
    \begin{align*}
        \lim_{x \to a^+} \frac{\varphi_{k}(x)}{\theta_{\ell_a+k}(x)} = 0, \qquad \lim_{x \to a^+} \frac{\theta_{\ell_a+k+1}(x)}{\varphi_k(x)} = 0,
    \end{align*}
    and 
    \begin{align*}
        \lim_{x \to a^+} \frac{\varphi_{k+1}(x)}{\varphi_{k}(x)} = 0, \qquad \lim_{x \to a^+} \frac{\theta_{k+1}(x)}{\theta_{k}(x)} = 0.
    \end{align*}
\end{theorem}
\begin{proof}
    Note that by \eqref{PhiRhoEst}, the definition of $\ell_a$, and \eqref{ThetaN+1Est} we have
    \begin{align*}
        \lim_{x \to a^+} \frac{\varphi_{1}(x)}{\varphi_{0}(x)} = 0, \qquad \lim_{x \to a^+} \frac{\varphi_{0}(x)}{\theta_{\ell_a}(x)} = 0,  \qquad  \lim_{x \to a^+} \frac{\theta_{\ell_a+1}(x)}{\varphi_{0}(x)} = 0.
    \end{align*}
    Together with \eqref{ThetaRatio} the theorem follows by Lemma~\ref{lemmaFF}.
\end{proof}
Note that Theorem \ref{Theorem2} can be summarized as stating that
\begin{align}\label{gg}
    |\theta_0(x)| \gg \dots \gg |\theta_{\ell_a}(x)| \gg |\varphi_0(x)| \gg |\theta_{\ell_a+1}(x)| \gg |\varphi_1(x)| \gg \dots,
\end{align}
where $f \gg g$ is shorthand for $\lim_{x \to a^+}\frac{g}{f}=0$. 

Note that in the more general setting of Corollary \ref{CorThetaNorm}, that is only assuming $W(\theta(z, \, \cdot \, ), \varphi(z, \, \cdot \,)) = 1$ and the standard assumption $\lim_{x\to a^+} \frac{\varphi(z_1,x)}{\varphi(z_2,x)} = 1$, the three limits
\begin{align} \label{threeLimits}
     \lim_{x \to a^+} \frac{\varphi_{k}(x)}{\varphi_m(x)}, \qquad \lim_{x \to a^+} \frac{\theta_{k}(x)}{\theta_m(x)}, \qquad \lim_{x \to a^+} \frac{\varphi_{k}(x)}{\theta_m(x)},  \qquad k, m \in \mathbb N_0,
\end{align}
always exist in the extended real numbers $\R \cup \lbrace \pm \infty \rbrace$ (use the previous theorem and Remark \ref{remarkexamples} $(i)$). This will be useful in the proof of Corollary \ref{cor:cannorm} which relies on L'H\^opital's rule.

Before continuing with an application to Darboux transforms, we show that the property of being a naturally normalized system is independent of the choice of $\lambda \in \R$ in Definition \ref{DefThetaNorm}. 

\begin{proposition}
The property of being a naturally normalized system is independent of the choice $\lambda \in \R$.
\end{proposition}
\begin{proof}
 Let $\varphi$, $\theta$ be naturally normalized for some fixed choice of $\lambda \in \R$. Choose any $\widetilde \lambda \in \R$ and write
    \begin{align*}
        \theta(z,x) = \sum_{n \geq 0} \widetilde \theta_n(x)(z-\widetilde \lambda)^n, \qquad \varphi(z,x) = \sum_{n \geq 0} \widetilde \varphi_n(x)(z-\widetilde \lambda)^n. 
    \end{align*}
We need to show that $\widetilde \theta_n$ satisfies
\begin{align}\label{tildeRatio}
    \lim_{x \to a^+} \frac{\widetilde \theta_n(x)}{\widetilde \varphi_0(x)} = 0 \ \text{ for}\ n > \ell_a.
\end{align}
To this end, observe that
\begin{align*}
    \widetilde \theta_n(x) = \sum_{j \geq n} \theta_j(x) \big( \widetilde \lambda - \lambda\big)^{j-n}\binom{j}{n}.
\end{align*}
Thus as $\lim_{x \to a^+}\frac{\widetilde \varphi_0(x)}{\varphi_0(x)} = 1$ by \eqref{phiNorm}, we have
\begin{align*}
    \lim_{x \to a^+} \frac{\widetilde \theta_n(x)}{\widetilde\varphi_0(x)} = \lim_{x \to a^+} \frac{\widetilde \theta_n(x)}{\varphi_0(x)} = \lim_{x \to a^+} \Bigg[\frac{\theta_n(x)}{\varphi_0(x)} + \sum_{j = n+1}^\infty \frac{\theta_j(x)}{\varphi_0(x)} \big( \widetilde \lambda - \lambda\big)^{j-n}\binom{j}{n}\Bigg].
\end{align*}
Now for $n > \ell_a$ we have $\lim_{x \to a^+} \frac{\theta_n(x)}{\varphi_0(x)} = 0$. In fact, $\big| \frac{\theta_j(x)}{\varphi_0(x)} \big| \leq \rho^{j -\ell_a-1}(x)$ for $j > \ell_a$ and $x \in (a,c)$ for some $c \in (a,b)$ by Lemma~\ref{LemThetaBound}. Hence, we can estimate
\begin{align*}
    \Bigg|\sum_{j = n+1}^\infty \frac{\theta_j(x)}{\varphi_0(x)} \big( \widetilde \lambda - \lambda\big)^{j-n}\hspace*{-5pt}\underbrace{\binom{j}{n}}_{\leq j^n/n!} \Bigg| \leq \rho(x) \underbrace{\frac{1}{n!}\sum_{j = n+1}^\infty \rho^{j - \ell_a -2}(x) \big| \widetilde \lambda - \lambda\big|^{j-n} j^n}_{= \Sigma(x)}, \;   x \in (a,c).
\end{align*}
Note that $j \geq \ell_a + 2$ in the above sum and $\Sigma(x)$ converges as long as $\rho(x)|\widetilde \lambda - \lambda| < 1$, which is true for $x$ close enough to $a$. Moreover, $\Sigma(x)$ is a monotonically increasing function (for $x$ increasing), while $\lim_{x \to a^+} \rho(x) = 0$. Thus $\lim_{x \to a^+} \rho(x)\Sigma(x) = 0$ and the claim \eqref{tildeRatio} follows finishing the proof.
\end{proof}

\section{Connection with Darboux transforms}\label{SectDarb}

Let us now turn to an application of the regularization index. For this we make additional regularity assumptions on our Sturm--Liouville differential expression.
\begin{hypothesis}\label{HypoDarb}
In addition to Hypothesis \ref{HypoInt}, assume further that $(pr),(pr)'/r \in AC_{loc}((a,b))$ and $(pr)\big|_{(a,b)}>0$.
\end{hypothesis}

With these assumptions the Sturm--Liouville differential expression \eqref{SL}  can be transformed into an equivalent Schr\"odinger differential expression given by
\begin{align}\label{SchroedingerForm}
    -\frac{d^2}{dX^2} + Q(X), \qquad X \in (A,B) \subseteq \mathbb R,
\end{align}
via the Liouville transform (see \cite{Ev05} and \cite[Sect. 3.5]{GNZ23}).
As mentioned in \cite[Sect. 3.5]{GNZ23}, we point out that the conditions given in Hypothesis \ref{HypoDarb} allow for different examples that are not included under the typical conditions assumed, namely, the conditions $p, p', r, r' \in AC_{loc}((a,b))$ with $p,r>0$ on $(a,b)$. This can be seen by considering the elementary example $p(x)=r(x)^{-1}=|x-1|^{1/2},\ q(x)=0$ on $(0,2)$ for instance.
Returning to applying the transform, choose $k \in (a,b)$ and define
\begin{align}\nonumber
    X(x) &= \int_k^x \sqrt{\frac{r(t)}{p(t)}} dt, \qquad x \in (a,b),
    \\\label{Def:AB}
     A &= -\int_a^k \sqrt{\frac{r(t)}{p(t)}} dt \ \in [-\infty, 0), \qquad B = \int_k^b \sqrt{\frac{r(t)}{p(t)}} dt \ \in (0, \infty],
     \\\nonumber
     Y(X) &= \big(p(x)r(x)\big)^{1/4}y(x), \qquad x \in (a,b),
     \\\nonumber
    Q(X) 
    &=\frac{q(x)}{r(x)}-\frac{1}{16p(x)r(x)}\left[\frac{(pr)'(x)}{r(x)}\right]^2+\frac{1}{4r(x)}\left[\frac{(pr)'(x)}{r(x)}\right]' , \qquad x \in (a,b).
\end{align}
Then $y$ solves $\tau y = z y$, if and only if
\begin{align*}
    -\frac{d^2}{dX^2}Y(z,X) + Q(X) Y(z,X) = z Y(z,X), \qquad X \in (A,B),
\end{align*}
where one readily verifies that under the assumptions in Hypothesis \ref{HypoDarb}, one has $Q\in L^1_{loc}((A,B);dX)$ where $dX=[r(x)/p(x)]^{1/2} dx$.
Note in particular that the regularization index remains invariant under the above transformation. 

The motivation for introducing Hypothesis \ref{HypoDarb} is twofold. First, defining Darboux transforms for general Sturm--Liouville differential expressions requires additional regularity assumptions on $p$ and $r$ (see \cite{GT96}). Secondly, Darboux transforms applied to differential expressions in Schr\"odinger form (also \emph{Liouville form}) have a much simpler form (cf.~\cite{G-UKM}).
Similar to the Liouville transform, Hypothesis \ref{HypoDarb} is weaker than the typical assumptions for Darboux transformations such as those found in \cite{GT96}, allowing for more general examples to be considered.

To avoid unnecessary notation, we will assume that $p \equiv 1$ and $r \equiv 1$, so that $\tau = -\frac{d^2}{dx^2}+q$ is already in Schr\"odinger form for the rest of this section.    

Let us now assume that $\psi$ is a positive solution of $\tau y = \lambda y$, meaning that 
\begin{align*}
    \tau \psi = \lambda \psi \ \text{ with} \ \psi(x) > 0, \quad  x \in (a,b).
\end{align*}
Such $\psi$, often called the seed function, exists if and only if $\tau$ is nonoscillatory at both endpoints, which we will assume from now on (see, e.g., \cite[Cor.~2.4]{GST}). Then as a formal differential expression, $\tau$ can be factorized as follows:
\begin{align*}
    \tau = -\frac{d^2}{dx^2} + q &= \Big(\frac{d}{dx} + \frac{\psi'}{\psi} \Big) \Big( -\frac{d}{dx} + \frac{\psi'}{\psi}\Big) +\lambda
    = B_\psi A_\psi +\lambda. 
\end{align*}
Note that we avoid the common notation $B_\psi = A_\psi^*$ to emphasize that $A_\psi$ and $B_\psi$ are just formal differential expressions rather than operators. We define the associated Darboux transformed differential expression by
\begin{align*}
    \widehat \tau = A_\psi B_\psi +\lambda &= \Big(-\frac{d}{dx} + \frac{\psi'}{\psi} \Big) \Big(\frac{d}{dx} + \frac{\psi'}{\psi}\Big) +\lambda
    = -\frac{d^2}{dx^2} + \widehat q,
\end{align*}
where 
\begin{align*}
    \widehat q = q -2 \frac{d}{dx}\Big( \frac{\psi'}{\psi}\Big),
\end{align*}
as can be verified by a direct computation. We say that $\widehat \tau$ is obtained from $\tau$ via a Darboux transform with seed function $\psi$.

Take two functions $f,g$ such that $\tau f = g$, and define $\widehat f = A_\psi  f$, $\widehat g = A_\psi g$. Then a quick computation shows that 
\begin{equation*}
\widehat \tau \widehat f = \big(A_\psi B_\psi +\lambda\big)A_\psi f = A_\psi \tau f = A_\psi g = \widehat g.
\end{equation*}
We want to study what happens if we apply $A_\psi$ to the naturally normalized system $\varphi$, $\theta$ defined in the previous sections (cf.~\cite[Sect.~3]{KST_MN}). Here we assume that Hypothesis \ref{Hypothesis} holds without any additional assumptions on the regularization index $\ell_a \in \mathbb N_0 \cup \lbrace \infty \rbrace$. We have to distinguish between two cases:
\\

\noindent
\textbf{Case 1: The seed function $\psi$ is principal at $a$}: In this case, possibly after scaling $\varphi$ by a real non-zero constant, we have 
\begin{align*}
    \varphi(z,x) = \psi(x) + \sum_{n \geq 1} \varphi_n(x)(z-\lambda)^n,
\end{align*}
that is,~$\psi = \varphi_0$, as principal solutions are unique up to scalar multiples. We define
\begin{align*}
    \widehat \varphi(z,x) = \frac{1}{z-\lambda} A_\psi \varphi(z,x) = \sum_{n \geq 0} \widehat \varphi_n(x)(z-\lambda)^n, \ \text{ where}\ \widehat   \varphi_n = A_\psi \varphi_{n+1}.
\end{align*}
Here we have used the analogue of Lemma \ref{LongLem} for $A_\psi$. Note that $A_\psi \psi = 0$. In a similar manner we define
\begin{align*}
    \widehat \theta(z,x) = A_\psi \theta(z,x) = \sum_{n \geq 0} \widehat \theta_n(x)(z-\lambda)^n, \ \text{ where}\ \widehat   \theta_n = A_\psi \theta_{n}.
\end{align*}
\\
\noindent
\textbf{Case 2: The seed function $\psi$ is nonprincipal at $a$}: In this case, possibly after scaling $\theta$ by a real non-zero constant and adding a real multiple of $\varphi$ to it, we have 
\begin{align*}
    \theta(z,x) = \psi(x) + \sum_{n \geq 1} \theta_n(x)(z-\lambda)^n,
\end{align*}
that is,~$\psi = \theta_0$. We then define
\begin{align}\label{wtNP}
    \widehat \theta(z,x) = \frac{1}{z-\lambda} A_\psi \theta(z,x) = \sum_{n \geq 0} \widehat \theta_n(x)(z-\lambda)^n, \ \text{ where } \widehat   \theta_n = A_\psi \theta_{n+1}.
\end{align}
In a similar manner, we define
\begin{align} \label{wphiNP}
    \widehat \varphi(z,x) = A_\psi \varphi(z,x) = \sum_{n \geq 0} \widehat \varphi_n(x)(z-\lambda)^n, \ \text{ where } \widehat   \varphi_n = A_\psi \varphi_{n}.
\end{align}

\subsection{Properties of \texorpdfstring{$\widehat \varphi$, $\widehat \theta$}{varphi, theta}}
Independent of the (non)principality of the seed function $\psi$, we have the equality 
\begin{align}\label{HatWronski}
    W\big(\wt(z, \, \cdot \,), \wphi(z, \, \cdot \,) \big) = 1.
\end{align}
This formula follows from the identity $W(A_\psi \theta, A_\psi \varphi \big) = (z-\lambda) W(\theta, \varphi)$, which can be verified by direct computation (note that the Wronskian $W(f,g) = fg'-f'g$ is the same for $\tau$ and $\widehat \tau$). 

Throughout this section we will repeatedly use L'H\^opital's rule as summarized in the following lemma.
\begin{lemma}\label{LemlH2}
    Let $f,g\in AC_{loc}((a,b))$ be given such that either $\lim_{x \to a^+} \frac{f(x)}{\psi(x)} = \lim_{x \to a^+} \frac{g(x)}{\psi(x)} = 0$ or $\lim_{x \to a^+} \frac{g(x)}{\psi(x)} = \pm \infty$. If the limit $\lim_{x \to a^+} \frac{A_\psi f(x)}{A_\psi g(x)}$ exists in the extended real numbers, then
    \begin{align*}
        \lim_{x \to a^+} \frac{f(x)}{g(x)} = \lim_{x \to a^+} \frac{A_\psi f(x)}{A_\psi g(x)} \in \mathbb R \cup \lbrace \pm \infty \rbrace.
    \end{align*}
\end{lemma}
\begin{proof}
    This is a simple application of L'H\^opital's rule:
    \begin{align*}
        \lim_{x \to a^+} \frac{f(x)}{g(x)} = \lim_{x \to a^+} \frac{\Bigg(\cfrac{f(x)}{\psi(x)}\Bigg)'}{\Bigg(\cfrac{g(x)}{\psi(x)}\Bigg)'} = \lim_{x \to a^+} \frac{A_\psi f(x)}{A_\psi g(x)}.
    \end{align*}
\end{proof}
We now proceed with the normalization properties of $\wphi$ and $\wt$. We begin with the following proposition.
\begin{proposition}\label{PropRatio}
    Assume Hypothesis \ref{HypoInt}, that Hypothesis \ref{Hypothesis} holds at $x=a$,  and let $\varphi$, $\theta$ be naturally normalized with either $\psi = \varphi_0$ or $\psi = \theta_0$. Then 
    \begin{enumerate}
        \item $\displaystyle\lim_{x \to a^+} \cfrac{\wphi(z_1,x)}{\wphi(z_2,x)} = 1$, \ for all \ $z_1, z_2 \in \mathbb R$, 
        \\
        \item $\displaystyle\lim_{x \to a^+} \cfrac{\wt(z_1,x)}{\wt(z_2,x)} = 1$, \ for all \ $z_1, z_2 \in \mathbb R$. 
    \end{enumerate}
    In particular, $\widehat \tau$ satisfies Hypothesis \ref{Hypothesis} at $x = a$.
\end{proposition}
\begin{proof}
First observe that $A_\psi f = \frac{W(f, \psi)}{\psi}$. Moreover, if $f$ is a solution of $\tau f = z f$, $z \not = \lambda$, we have
    \begin{align}\label{diffW}
        \frac{d}{dx}\big( W(f, \psi)(x) \big) = (z-\lambda) f(x) \psi(x), 
    \end{align}
    which implies that $W(f,\psi)(x)$ is monotonic as $x \to a^+$, as $\tau y = zy$ is nonoscillatory for all $z \in \mathbb R$. It then follows that $A_\psi f = \frac{W(f, \psi)}{\psi}$ is nonoscillatory, implying that $\widehat \tau y = z y$ is nonoscillatory at $a$ for all $z \in \mathbb R \setminus \lbrace \lambda \rbrace$. It is also nonoscillatory for $z = \lambda$, as being nonoscillatory for $\lambda_1$ implies being nonoscillatory for any $\lambda_2 < \lambda_1$.
    
    Note the limit $\displaystyle\lim_{x \to a^+} \tfrac{\wphi(z_1,x)}{\wphi(z_2,x)}$ exists in the extended real numbers as $\big(\frac{\wphi(z_1,x)}{\wphi(z_2,x)}\big)' = \frac{W(\wphi(z_2,x), \wphi(z_1,x))}{\wphi(z_2,x)^2}$ is nonoscillatory near $a$ by \eqref{diffW}. In case $\psi$ is nonprincipal, we obtain using L'H\^opital's rule,
    \begin{align*}
        1 = \lim_{x \to a^+} \frac{\varphi(z_1,x)}{\varphi(z_2,x)} = \lim_{x \to a^+} \frac{\Bigg(\cfrac{\varphi(z_1,x)}{\psi(x)}\Bigg)'}{\Bigg(\cfrac{\varphi(z_2,x)}{\psi(x)}\Bigg)'} = \lim_{x \to a^+} \frac{A_\psi \varphi(z_1,x)}{A_\psi \varphi(z_2,x)} =  \lim_{x \to a^+} \frac{\wphi(z_1,x)}{\wphi(z_2,x)}.
    \end{align*}
    If $\psi$ is principal we instead write
    \begin{align*}
        1 = \lim_{x \to a^+} \frac{\theta(z_1, x)}{\theta(z_2, x)} = \lim_{x \to a^+} \frac{\Bigg(\cfrac{\theta(z_1, x)}{\psi(x)}\Bigg)'}{\Bigg(\cfrac{\theta(z_2, x)}{\psi(x)}\Bigg)'}
        =  \lim_{x \to a^+} \frac{A_\psi \theta(z_1,x)}{A_\psi \theta(z_2,x)}=\lim_{x \to a^+} \frac{\wt(z_1,x)}{\wt(z_2,x)}.
    \end{align*}

   Now the asymptotics $(i)$ and $(ii)$ are equivalent by Corollary \ref{CorThetaNorm} together with \eqref{HatWronski}, finishing the proof. 
\end{proof}

The next results concern the (non)principality of $\wphi$, $\wt$.
\begin{proposition}\label{PropPhiTheta}
    Assume Hypothesis \ref{HypoInt}, that Hypothesis \ref{Hypothesis} holds at $x=a$,  and let $\varphi$, $\theta$ be naturally normalized such that additionally either $\psi = \varphi_0$ or $\psi = \theta_0$, where $\psi > 0$ is the seed function.
    \begin{enumerate}
        \item If the seed function $\psi$ is nonprincipal and $\tau$ is in the limit circle case at $a$, then $\wphi(z,x)$ is nonprincipal and $\wt(z,x)$ is principal at $x=a$ for all $z \in \mathbb R$.
        \\
        \item In all other cases $\wphi(z,x)$ is principal and $\wt(z,x)$ is nonprincipal at $x=a$ for all $z \in \mathbb R$
    \end{enumerate}
\end{proposition}
\begin{proof}
    $(i)$ By assumption, $\tau$ is in the limit circle case at $x=a$ so that $\ell_a = 0$ by Theorem \ref{Thm:N=0}. We know from the proof of Proposition \ref{PropRatio} that $\wphi$, $\wt$ are nonoscillatory at $a$, meaning that the limit $\lim_{x \to a^+} \frac{\wt_0(x)}{\wphi_0(x)}$ must exist in the extended real numbers (as $(\widehat \theta_0/\widehat \varphi_0)'= -\widehat\varphi^{-2}_0$). Hence, using L'H\^opital's rule and definitions \eqref{wtNP}, \eqref{wphiNP}, as $\psi$ is assumed to be nonprincipal, we obtain 
    \begin{align*}
        0 = \lim_{x \to a^+} \frac{\theta_1(x)}{\varphi_0(x)} = \lim_{x \to a^+} \frac{\Bigg(\cfrac{\theta_1(x)}{\psi(x)}\Bigg)'}{\Bigg(\cfrac{\varphi_0(x)}{\psi(x)}\Bigg)'}= \lim_{x \to a^+} \frac{A_\psi \theta_1(x)}{A_\psi \varphi_0(x)} =  \lim_{x \to a^+} \frac{\wt_0(x)}{\wphi_0(x)}.
    \end{align*}
    Note that $\lim_{x \to a^+} \frac{\theta_1(x)}{\psi(x)} = \lim_{x \to a^+} \frac{\varphi_0(x)}{\psi(x)} = 0$, allowing us to use L'H\^opital. Hence, $\wt$ is principal and $\wphi$ is nonprincipal for all $z \in \C$ by Proposition \ref{PropRatio}. 

    $(ii)$ This case is essentially identical to the previous one with the roles of $\varphi$ and $\theta$ interchanged. First let us assume that $\psi$ is nonprincipal, that is,~$\psi = \theta_0$. Then necessarily $a$ is in the limit point case as otherwise, we are in the previous case $(i)$. Thus we have $\ell_a \geq 1$ and we obtain
    \begin{align*}
        0 = \lim_{x \to a^+} \frac{\varphi_0(x)}{\theta_1(x)} = \lim_{x \to a^+} \frac{\Bigg(\cfrac{\varphi_0(x)}{\psi(x)}\Bigg)'}{\Bigg(\cfrac{\theta_1(x)}{\psi(x)}\Bigg)'}= \lim_{x \to a^+} \frac{A_\psi \varphi_0(x)}{A_\psi \theta_1(x)} = \lim_{x \to a^+} \frac{\wphi_0(x)}{\wt_0(x)}.
    \end{align*}
    In case $\psi$ is principal, that is,~$\psi = \varphi_0$ we obtain
    \begin{align*}
        0 = \lim_{x \to a^+} \frac{\varphi_1(x)}{\theta_0(x)} = \lim_{x \to a^+} \frac{\Bigg(\cfrac{\varphi_1(x)}{\psi(x)}\Bigg)'}{\Bigg(\cfrac{\theta_0(x)}{\psi(x)}\Bigg)'}
= \lim_{x \to a^+} \frac{A_\psi \varphi_1(x)}{A_\psi \theta_0(x)} = \lim_{x \to a^+} \frac{\wphi_0(x)}{\wt_0(x)}.
    \end{align*}
    Similarly to before, it follows from Proposition \ref{PropRatio} that $\wt$ is nonprincipal and $\wphi$ is principal for all $z \in \C$. 
\end{proof}
As a corollary we now obtain the following.
\begin{corollary}\label{cor:cannorm}
     Assume Hypothesis \ref{HypoInt}, that Hypothesis \ref{Hypothesis} holds at $x=a$,  and let $\varphi$, $\theta$ be naturally normalized with either $\psi = \varphi_0$ or $\psi = \theta_0$. Then 
     \begin{align*}
         \lim_{x \to a^+} \frac{\wphi_{n+1}(x)}{\wphi_n(x)} = 0, \qquad \lim_{x \to a^+} \frac{\wt_{n+1}(x)}{\wt_n(x)} = 0, \qquad n \in \mathbb N_0.
     \end{align*}
     In particular, the system $\widehat \theta$, $\widehat \varphi$ is naturally normalized.
\end{corollary}
\begin{proof}
First let us assume that $\psi$ is principal, that is,~$\psi= \varphi_0$. Then we compute using \eqref{threeLimits} and Lemma \ref{LemlH2},
\begin{align*}
    0=\lim_{x \to a^+} \frac{\varphi_{n+2}(x)}{\varphi_{n+1}(x)} = \lim_{x \to a^+} \frac{\Bigg(\cfrac{\varphi_{n+2}(x)}{\psi(x)}\Bigg)'}{\Bigg(\cfrac{\varphi_{n+1}(x)}{\psi(x)}\Bigg)'} = \lim_{x \to a^+} \frac{A_\psi \varphi_{n+2}(x)}{A_\psi \varphi_{n+1}(x)} =  \lim_{x \to a^+} \frac{\wphi_{n+1}(x)}{\wphi_n(x)},
\end{align*}
and 
\begin{align*}
    0=\lim_{x \to a^+} \frac{\theta_{n+1}(x)}{\theta_{n}(x)} = \lim_{x \to a^+} \frac{\Bigg(\cfrac{\theta_{n+1}(x)}{\psi(x)}\Bigg)'}{\Bigg(\cfrac{\theta_{n}(x)}{\psi(x)}\Bigg)'} = \lim_{x \to a^+} \frac{A_\psi \theta_{n+1}(x)}{A_\psi \theta_{n}(x)} =  \lim_{x \to a^+} \frac{\wt_{n+1}(x)}{\wt_n(x)}.
\end{align*}
Note that for $n \leq \ell_a$ we indeed have $\lim_{x \to a^+} \frac{\theta_{n}(x)}{\psi(x)} = \pm \infty$, while for $n > \ell_a$ we have $\lim_{x \to a^+} \frac{\theta_{n}(x)}{\psi(x)} = \lim_{x \to a^+} \frac{\theta_{n+1}(x)}{\psi(x)} = 0$, allowing us to use Lemma~\ref{LemlH2}.

The case of $\psi$ nonprincipal, that is,~$\psi= \theta_0$ can be shown analogously.
\end{proof}

\section{Regularization via Darboux transforms}\label{SectSpec}
The goal of the present section is to demonstrate how applying a series of Darboux transforms can, in certain cases, be viewed as a type of regularization procedure. More precisely, we will say that a Schr\"odinger differential expression $\tau$ is `regularizable via Darboux transforms at $a$' if and only if there exists a finite sequence of Darboux transforms 
\begin{align*}
    \tau\to \widehat \tau_1 \to \dots \to \widehat \tau_N = \widetilde \tau
\end{align*}
 such that $\widetilde \tau$ is in the limit circle case at $x = a$.

It turns out that being regularizable via Darboux transforms at $a$ is equivalent to $\ell_a$ being finite (see Theorem \ref{Thm:Transformation} below). The key to this observation is the following theorem which shows that the regularization index is well-behaved under Darboux transforms. For this, recall that $\widehat \tau$ will satisfy Hypothesis \ref{Hypothesis} in case $\tau$ satisfies it (see Proposition  \ref{PropRatio}), and thus will have a well-defined regularization index $\widehat \ell_a\in \mathbb{N}_0 \cup \lbrace \infty \rbrace$.
\begin{theorem}\label{Thm:Darboux}
    Assume Hypothesis \ref{HypoInt} and that Hypothesis \ref{Hypothesis} holds for $\tau$ at the endpoint $x=a$. If $\ell_a$ is the regularization index of $\tau$ at $x=a$ then the regularization index $\widehat \ell_a$ of $\widehat \tau$ at $x=a$ satisfies $($where we interpret $\infty\pm1$ as $\infty$$)$
    \begin{enumerate}
        \item $\widehat \ell_a = \ell_a + 1$ if the seed function $\psi$ is principal at $x=a$,
        \\
        \item and \begin{align*}
            \widehat \ell_a = \begin{cases}
                0, & \text{if } \ell_a = 0,
                \\
                \ell_a - 1, & \text{if } \ell_a > 0,
            \end{cases}
         \end{align*}
            if the seed function $\psi$ is nonprincipal at $x=a$.

    \end{enumerate}
\end{theorem}
\begin{proof}
    Let us choose a naturally normalized system $\varphi$, $\theta$ and let $k > 0$, $m \geq 0$ if $\psi$ is principal, and $k \geq 0$, $m > 0$ if $\psi$ is nonprincipal. Then using L'H\^opital's rule we can compute
    \begin{align*}
        \lim_{x \to a^+} \frac{\varphi_{k}(x)}{\theta_{m}(x)} = \lim_{x \to a^+} \frac{\Bigg(\cfrac{\varphi_{k}(x)}{\psi(x)}\Bigg)'}{\Bigg(\cfrac{\theta_{m}(x)}{\psi(x)}\Bigg)'} = \lim_{x \to a^+} \frac{A_\psi \varphi_{k}(x)}{A_\psi \theta_{m}(x)} =  \lim_{x \to a^+} \frac{\wphi_{k-\delta_1}(x)}{\wt_{m-\delta_2}(x)},
    \end{align*}
    where $\delta_1 = 1$, $\delta_2  = 0$ if $\psi$ is principal, and $\delta_1 = 0$, $\delta_2  = 1$ if $\psi$ is nonprincipal. Note that the requirements for $k, m$ guarantee that we are in the setting of Lemma~\ref{LemlH2}. Thus, in case $\wt$ is nonprincipal, the regularization index of $\widehat \tau$ is given by $\widehat \ell_a = \ell_a+\delta_1-\delta_2$ by Definition \ref{DefRI}. In case $\wt$ is principal (so $\wphi$ is nonprincipal), which can only happen if $\psi$ is nonprincipal and $\ell_a = 0$, we have $\lim_{x \to a^+} \frac{\wphi_{0}(x)}{\wt_{0}(x)} = \lim_{x \to a^+} \frac{\varphi_{0}(x)}{\theta_{1}(x)} = \pm \infty$ and $\lim_{x \to a^+} \frac{\wphi_{1}(x)}{\wt_{0}(x)} = \lim_{x \to a^+} \frac{\varphi_{1}(x)}{\theta_{1}(x)} = 0$. Thus $\widehat \ell_a = 0$, finishing the proof.
\end{proof}
We can now prove the following result stating that Hypothesis \ref{Hypothesis} with $\ell_a < \infty$ is necessary to transform a Schr\"odinger differential expression into one which is in the limit circle case at $a$ through a finite number of Darboux transforms. This then gives us a complete characterization of those Schr\"odinger differential expressions which can be regularized via Darboux transforms. 

\begin{theorem}\label{Thm:Transformation}
    Assume Hypothesis \ref{HypoInt} and let $\tau$ be a Schr\"odinger differential expression which is nonoscillatory at both endpoints. Then $\tau$ can be transformed via a finite series of Darboux transforms to a Schr\"odinger differential expression $\widetilde \tau$ which is in the limit circle case at $x = a$ if and only if Hypothesis \ref{Hypothesis} holds at $x=a$ and $\ell_a < \infty$. In this case, the minimal number of Darboux transforms is $\ell_a$ and is achieved if the seed functions are always chosen to be nonprincipal at $x = a$.
\end{theorem}
\begin{proof}
    First, let us remark that as $\tau$ is assumed to be nonoscillatory at both endpoints, nonvanishing seed functions can always be found. Indeed, as we will see in \eqref{seedFunction}, one can explicitly write down such seed functions which are nonprincipal at both endpoints. 
    Hence, it follows from Theorem \ref{Thm:Darboux} that if $\ell_a < \infty$ there exists a sequence of $\ell_a$ many Darboux transforms leading to a $\widetilde \tau$ which is in the limit circle case at $x = a$. It is also immediate from Theorem \ref{Thm:Darboux} that no smaller number of Darboux transforms will achieve this.
    
    It remains to show that if Hypothesis \ref{Hypothesis} does not hold for $\tau$, then it also does not hold for its Darboux transform $\widehat \tau$.
    
Let us now assume that the seed function $\psi$ is principal (the other case is proven in a similar manner). Then using L'H\^opital's rule we obtain
    \begin{align*}
        \lim_{x \to a^+} \frac{\theta(z_1,x)}{\theta(z_2,x)} = \lim_{x \to a^+} \frac{\Bigg(\cfrac{\theta(z_1, x)}{\psi(x)}\Bigg)'}{\Bigg(\cfrac{\theta(z_2, x)}{\psi(x)}\Bigg)'} = \lim_{x \to a^+} \frac{A_\psi \theta(z_1, x)}{A_\psi \theta(z_2, x)} =  \lim_{x \to a^+} \frac{\wt(z_1, x)}{\wt(z_2, x)}.
    \end{align*}
 Now it follows from Theorem \ref{TFAE} that if Hypothesis \ref{Hypothesis} does not hold for $\tau$, it will also not hold for $\widehat \tau$. By induction, the same is true for any $\widetilde \tau$ obtained from $\tau$ through a finite series of Darboux transform, implying that $\widetilde \tau$ must remain in the limit point case at $x=a$. This finishes the proof.
\end{proof}
As a simple corollary we also obtain.
\begin{corollary}\label{CorDarb}
    A Schr\"odinger differential expression $\tau$ which is nonoscillatory at both endpoints can be transformed $($via Darboux transforms$)$ to one which is in the limit circle case at both endpoints if and only if Hypothesis \ref{Hypothesis} holds and $\ell_a, \ell_b < \infty$. In this case the minimal number of Darboux transforms is $\max\{\ell_a, \ell_b\}$ and is achieved by choosing seed functions which are nonprincipal at both endpoints.
\end{corollary}
In other words, $\tau$ is regularizable via Darboux transforms at both endpoints if and only if $\ell_a, \ell_b < \infty$.

We should remark that while these results are proven for Schr\"odinger differential expressions, analogous statements can be made for general Sturm--Liouville differential expressions satisfying Hypothesis \ref{HypoDarb} through the use of the Liouville transform which leaves the regularization index invariant.

To justify the terminology `regularization via Darboux transforms' we recall that Darboux transforms as above can change the spectrum only at the value of the spectral parameter of the seed function, provided correct boundary conditions are specified (see \cite{D78}). In particular, self-adjoint realizations of $\tau$ and self-adjoint realizations of its Darboux transform $\widehat \tau$ will have in general similar spectral properties. Thus we expect problems having finite regularization indices to behave similarly to regular problems. As an example, we say that $\tau$ satisfies Weyl asymptotics if and only if every self-adjoint realization $T$ of $\tau$ has a discrete spectrum  such that the eigenvalues additionally satisfy the asymptotics
\begin{equation}\label{eq:Weyl}
\lambda_{n}\underset{n\to\infty}{\sim} \pi^2 n^2 \left(\int_{a}^{b}\sqrt{\frac{r(t)}{p(t)}}dt\right)^{-2}, \qquad \sigma(T) = \lbrace \lambda_n \rbrace_{n=1}^{\infty}.
\end{equation}
Here it is implicitly assumed that the integral above is finite.

It is known that $\tau$ satisfies Weyl asymptotics if it is regular, or more generally limit circle nonoscillatory at both endpoints (see \cite[Remark 3.1, Lem.~3.5(3)]{NZ92}). Hence, we obtain from Corollary \ref{CorDarb} an elementary proof of the following fact.
\begin{corollary}
Assume Hypotheses \ref{HypoInt} and \ref{HypoDarb}. If Hypothesis \ref{Hypothesis} is satisfied at $x=a, b$ with $\ell_a,\ell_b<\infty$, then the Weyl asymptotics \eqref{eq:Weyl} hold.
\end{corollary}
\begin{proof}
Note that given $\tau$ satisfying Hypothesis \ref{HypoDarb} we can transform it to an equivalent differential expression $\tau_{SF} = -\frac{d^2}{dX^2} + Q(X), \ X \in (A,B) \subseteq \mathbb R$
in  Schr\"odinger form via the Liouville transform. A similar isospectral transformation $T \to T_{SF}$ with $\sigma(T) = \sigma(T_{SF})$ holds for self-adjoint realizations of $\tau$ (see, e.g., \cite[Sect. 3.2]{FGKS21} for the regular case). From Prop.~\ref{Prop:Nevanlinna2}, and as the interval $(A,B)$ remains invariant under Darboux transforms, it follows that the eigenvalues of $T_{SF}$ (hence of $T$), satisfy $\lambda_n \underset{n\to\infty}{\sim} n^2 \pi^2(B-A)^{-2}$, in particular $B-A$ must be finite. But $B-A = \int_a^b \sqrt{\frac{r(t)}{p(t)}} dt$ according to \eqref{Def:AB}. This shows \eqref{eq:Weyl} finishing the proof. 
\end{proof}
We remark that the above result is not new. In fact \cite[p.~33]{WW14} implies that, provided the regularization indices are finite, Weyl asymptotics hold even without Hypothesis \ref{HypoDarb}. This result is based on \cite[Thm.~4.8]{W11}, which appears to use very different techniques compared to the present paper.  

\begin{proposition}{\emph{(}\cite[Thm.~4.8]{W11}, \cite{WW14}\emph{)}}
Assume Hypothesis \ref{HypoInt}. If Hypothesis \ref{Hypothesis} is satisfied at $x=a, b$ with $\ell_a, \ell_b <\infty$, then the Weyl asymptotics \eqref{eq:Weyl} hold.
\end{proposition}
\begin{proof}
This result follows from \cite[p.~33]{WW14} since $\ell_a, \ell_b <\infty$ implies $\Delta$ used in \cite{LW23} is also finite (see Remark \ref{RemarkDelta} for instance).
\end{proof}
As our regularization procedure relies on the Liouville and Darboux transforms, Hypothesis \ref{HypoDarb} was necessary. This raises the following question:

\begin{problem}
Assume Hypothesis \ref{HypoInt}. If Hypothesis \ref{HypoDarb} does not hold, but Hypothesis \ref{Hypothesis} is satisfied at $x=a, b$ with $\ell_a,\ell_b<\infty$ and at least one index positive, can the problem still be regularized $($in the sense of transforming into an associated regular problem$)$ using a different method?
\end{problem}

Regarding Weyl asymptotics, our work relies heavily on the finiteness of the regularizations indices. However, having an infinite regularization index is still compatible with Weyl asymptotics $($see Section \ref{MiePotential} for an example$)$. This leads us to the following problem:
\begin{problem}
Can one characterize when Weyl asymptotics will hold based on the behavior of the system $\varphi,\theta$ for problems with infinite index?
\end{problem}
We point out that Sections \ref{MiePotential} and \ref{subPoly} show that for an infinite regularization index, Weyl-asymptotics may or may not hold. In such cases we cannot rely on Darboux transforms, so we instead use the following lemma to prove Weyl asymptotics for the example in Section \ref{MiePotential} which has an infinite regularization index.

\begin{lemma}\label{LemmaWeyl}
    Let $a,b$ be finite and assume that $\tau$ satisfies Weyl asymptotics. If the potential $q_1$ is bounded from below, then $\tau_{q_1}:=\tau+\frac{q_1(x)}{r(x)}$ satisfies Weyl asymptotics.
\end{lemma}
\begin{proof}
    Fix $p,$ $r,$ and $q$, and assume w.l.o.g that $q_1 \geq 0$.
    Let $\lambda_n^{q_1,\varepsilon}$ be the Dirichlet eigenvalues of $\tau_{q_1}|_{(a+\varepsilon, b-\varepsilon)}$ for $\varepsilon > 0$ small enough. Then these eigenvalues will satisfy Weyl asymptotics, $\lambda_n^{q_1, \varepsilon} = n^2 \pi^2\left(\int_{a+\varepsilon}^{b-\varepsilon}\sqrt{\frac{r(t)}{p(t)}}dt\right)^{-2}(1+o(1))$ as $n \to \infty$, as the problem is regular at both endpoints. Similarly, the Dirichlet eigenvalues $\lambda_n^{\varepsilon}$ of $\tau|_{(a+\varepsilon, b - \varepsilon)}$ also satisfy the same asymptotics.

    By the Sturm--Picone comparison theorem we must have that $\lambda_n^{ \varepsilon} \leq \lambda_n^{q_1,\varepsilon}$ for all $n \geq 0$. Moreover \cite[Thm.~4.4.4]{Ze05} (and the remark after it) shows that the Dirichlet eigenvalues $\lambda_n^{\varepsilon}$ and $\lambda_n^{q_1,\varepsilon}$ increase as a function of $\varepsilon$, while \cite[Thm.~10.8.2]{Ze05} shows that $\lim_{\varepsilon \to 0} \lambda_n^{q_1,\varepsilon} = \lambda_n^{q_1}$ resp.~$\lim_{\varepsilon \to 0} \lambda_n^{\varepsilon} = \lambda_n$, where $\lambda_n^{q_1}$ and $\lambda_n$ are the eigenvalues of the Friedrichs realization of $\tau_{q_1}$ and $\tau$, respectively. Note that as $\tau$ has Weyl asymptotics by assumption, we have $\lambda_n = n^2 \pi^2\left(\int_{a}^{b}\sqrt{\frac{r(t)}{p(t)}}dt\right)^{-2}(1+o(1))$ as $n \to \infty$. Thus it follows that $\lambda_n =  \lim_{\varepsilon' \to 0}\lambda^{\varepsilon'}_n \leq \lambda_n^{q_1} \leq \lambda_n^{q_1,\varepsilon}$ for all $\varepsilon > 0$, proving Weyl asymtptotics for $\tau_{q_1}$.
\end{proof}

\section{Weyl \texorpdfstring{$m$}{m}-function in case of finite regularization index}\label{SectWeylm}
In this section we compute Weyl $m$-functions in the case of a finite regularization index. As it turns out, these $m$-functions will be in the generalized Nevanlinna--Herglotz class $N_\kappa^\infty$ with $\kappa = \floor{(\ell+1)/2}$. More explicitly, the underlying $m$-functions can be written in terms of $m$-functions of limit circle problems having the familiar Nevanlinna--Herglotz property (see Propositions \ref{Prop:Nevanlinna1} and \ref{Prop:Nevanlinna2}).

Again to simplify our analysis, we assume that $\tau$ is in the Schr\"odinger form \eqref{SchroedingerForm} with $\ell_a, \ell_b < \infty$. As we are interested in at least one endpoint being in the limit point case, we will exclude the case of both endpoints being limit circle and without loss of generality assume $0\leq \ell_a\leq \ell_b < \infty$ with $\ell_b\geq1$.

Assuming momentarily that $\ell_a\geq 1$, we denote by $T$ the unique self-adjoint realization of $\tau$ (as both endpoints are in the limit point case). Consider now an arbitrary naturally normalized system $\theta_a$, $\varphi_a$ at the endpoint $x = a$ (and analogously $\theta_b$, $\varphi_b$ at the endpoint $x = b$). Then we can define the singular Weyl $m$-function (see \cite{GZ06}, \cite{KST_IMRN}) satisfying the equation
\begin{align}\label{mFunction}
    \theta_a(z,x) + m_T(z)\varphi_a(z,x) = D(z)\varphi_b(z,x), \qquad x \in (a,b), \quad z \in \C \setminus \R,
\end{align}
where $D(\,\cdot\,)$ is some holomorphic function on $\C \setminus \R$.
It is known that the spectrum of $T$ can be recovered from the limits $\lim_{\delta \to 0\pm} m_T(x + i\delta)$, $x \in \R$. In particular, in the presence of a purely discrete spectrum, $m_T$ can be extended to a meromorphic function with poles at the eigenvalues of $T$.

As described in the previous section, we will `regularize' the expression $\tau$ by applying a sequence of Darboux transforms which lower the regularization index $\ell_a$ to zero so that $x=a$ is a limit circle endpoint while simultaneously lowering $\ell_b$. While this process will not be unique, we will choose the optimal method in the sense of the least number of transforms while also adding the fewest eigenvalues possible. To this end we are looking for a seed function $\psi$ solving $\tau\psi = \lambda \psi$ satisfying the following properties:
\begin{enumerate}
    \item $\psi(x) \not =  0$ for all $x \in (a,b);$
    \item $\psi$ is nonprincipal at both endpoints.
\end{enumerate}
We choose an appropriate seed function $\psi_1$ satisfying conditions $(i)$, $(ii)$ above as follows. Consider any $\lambda_1 < \inf \sigma(T)$ (note that $T$ is bounded from below as $\tau$ is nonoscillatory at both endpoints). Then $\varphi_a(\lambda_1, x)$ is nonprincipal at $x = b$, as otherwise $\lambda$ would be an eigenvalue of $T$ (recall that Hypothesis \ref{Hypothesis} implies that principal solutions are $L^2$-integrable near the endpoint in question). Moreover, it follows from \cite[Cor.~2.4]{GST} that $\varphi_a(\lambda_1, x)$ has a fixed sign on $(a,b)$. We now define the seed function as
\begin{align}\label{seedFunction}
\psi_1(x) =\varphi_a(\lambda_1,x)\Bigg[\int_x^{b} \frac{dt}{p(t)\varphi^2_a(\lambda_1,t)} + 1 \Bigg].
\end{align}
It is easy to see that conditions $(i)$, $(ii)$ are satisfied. Moreover, we add a constant multiple $C_1=C_1(\lambda_1)\in\R$ of $\varphi_a$ to $\theta_a$ in order to also have
\begin{equation*}
\psi_1(x) =\theta_{a,1}(\lambda_1,x)=\theta_a(\lambda_1, x)+C_1\varphi_a(\lambda_1, x).
\end{equation*}
Introducing the notation $\varphi_{c,1}=\varphi_c,\ c\in\{a,b\}$, we arrive at the equation
\begin{align}\label{mFunction2}
    \theta_{a,1}(z,x) + (m_T(z)-C_1)\varphi_{a,1}(z,x) = D(z)\varphi_{b,1}(z,x), \qquad x \in (a,b), \quad z \in \C \setminus \R.
\end{align}
This corresponds to changing $m_T$ in \eqref{mFunction} by an additive constant. 

Applying $\frac{1}{z-\lambda_1}A_{\psi_1}$ to both sides of \eqref{mFunction2} yields
\begin{align*}
    \widehat \theta_{a,1}(z,x) + \underbrace{\frac{m_T(z)-C_1}{z-\lambda_1}}_{m_{\widehat T_1}(z)} \widehat \varphi_{a,1}(z,x) = \frac{1}{z-\lambda_1} D(z) \widehat \varphi_{b,1}(z,x), \quad x \in (a,b), \ z \in \C \setminus \R.
\end{align*}
Note that $m_{\widehat T_1}$ is the $m$-function of the self-adjoint realization $\widehat T_1$ of the Darboux transformed $\widehat \tau_1$, with Dirichlet boundary conditions (if any) since $\widehat \varphi_{a,1}$ and $\widehat \varphi_{b,1}$ are principal at $a$ and $b$, respectively. Moreover, as $\theta_{a,1}(\lambda_1,x) = \psi_1(x)$ is nonprincipal at $x = b$, it follows that $\lim_{\delta \to 0\pm} m_T(\lambda_1+i \delta) - C_1 \not = 0$, implying that $\sigma(\widehat T_1) = \sigma(T) \cup \lbrace \lambda_1 \rbrace$, that is, one eigenvalue was added to the spectrum. Also, by Theorem \ref{Thm:Darboux}, we have that both indices have been lowered by 1.

Now, since we assumed $\ell_a\leq\ell_b$, we can repeat this procedure $\ell_a$ times as follows: we introduce the notation $\varphi_{c,j+1}=\widehat\varphi_{c,j},\ c\in\{a,b\},\ j\in\N,$ and shift by a constant $C_{j+1}=C_{j+1}(\lambda_{j+1})\in\R$ after choosing $\lambda_{j+1}<\lambda_{j}$ at each step through
\begin{align*}
\theta_{a,j+1}(z,x)=\widehat \theta_{a,j}(z,x)+C_{j+1}\widehat\varphi_{a,j}(z,x),\qquad x\in(a,b),\quad z\in\C\setminus\R,\\
j\in\{1,2,\dots,\ell_{a}-1\}.
\end{align*}
Hence, after $\ell_a$ steps one arrives at 
\begin{align}\label{hatmfunctioneq}
    \widehat \theta_{a,\ell_a}(z,x) + m_{\widehat T_{\ell_a}}(z) \widehat \varphi_{a,\ell_a}(z,x) = \prod_{j=1}^{\ell_a}(z-\lambda_j)^{-1} D(z) \widehat \varphi_{b,\ell_a}(z,x), \notag \\
     x \in (a,b), \quad z \in \C \setminus \R,
\end{align}
with
\begin{equation}\label{hatmFunction1}
m_{\widehat T_{\ell_a}}(z)=\Bigg[\prod_{j=1}^{\ell_a}(z-\lambda_j)^{-1}\Bigg] \left[m_{T}(z)-\sum_{k=1}^{\ell_a}C_k\prod_{n=1}^{k-1}(z-\lambda_n)\right],\qquad z\in\C\setminus\R,
\end{equation}
where $\lambda_j$ denotes the eigenvalue added (below the previous step's spectrum) at the $j$th step with $j\in\{1,2,\dots,\ell_a\}$ and the product is empty and equal to $1$ for $k=1$. If $\ell_a=0$, then the products in \eqref{hatmFunction1} are considered empty and equal to 1 (i.e., no Darboux transformation was needed to make the endpoint $a$ limit circle so \eqref{hatmfunctioneq} and \eqref{mFunction} agree). As $\widehat \varphi_{a,\ell_a}$ and $ \widehat \varphi_{b,\ell_a}$ remain principal at $a$ and $b$, respectively (see Proposition \ref{PropPhiTheta}), $m_{\widehat T_{\ell_a}}$ is the $m$-function of the self-adjoint realization $\widehat T_{\ell_a}$ of $\widehat\tau_{\ell_a}$ with necessarily Dirichlet boundary conditions at $x=a$ (as $\widehat\ell_{a_{\ell_a}}=0$) and either Dirichlet at $x=b$ if $\ell_a=\ell_b$ or no boundary conditions at $x=b$ if $\ell_b>\ell_a$ (as $\widehat\ell_{b_{\ell_a}}=\ell_b-\ell_a$). That is, $\widehat T_{\ell_a}$ is the Friedrichs realization of $\widehat \tau_{\ell_a}$ (see \cite[Thm. 4.7]{GLN20}).

As it is known that when $a$ is a limit circle endpoint, the Weyl $m$-function defined via \eqref{hatmFunction1} with a naturally normalized system of solutions is in fact a Nevanlinna--Herglotz function (also called a Pick function; see \cite[Eq. (5.12)]{GLN20}), we arrive at the following result (see Remark \ref{remarkregularization} for weaker conditions for the theorem to hold):

\begin{proposition}\label{Prop:Nevanlinna1}
Assume Hypotheses \ref{HypoInt} and \ref{HypoDarb}. Furthermore, assume Hypothesis \ref{Hypothesis} is satisfied at $x=a, b$ with $\ell_a,\ell_b<\infty$, let $\ell=\min\{\ell_a,\ell_b\}$, and denote by $T$ either the unique self-adjoint realization of $\tau$ if $\ell_a,\ell_b\neq0$ or the Friedrichs realization otherwise. Then for each choice of $\ell$ real numbers satisfying $\lambda_\ell<\dots<\lambda_{1}<\inf\sigma(T)$ there is a Nevanlinna--Herglotz function $m_{\widehat T_\ell}$ and $\ell$ constants $C_j=C_j(\lambda_j)\in\R$, $j = 1, \dots, \ell$ such that the Weyl $m$-function $m_T$ in \eqref{mFunction} can be written as 
\begin{equation*}
m_{T}(z)=\Bigg[\prod_{j=1}^{\ell}(z-\lambda_j)\Bigg]m_{\widehat T_{\ell}}(z)+\sum_{k=1}^{\ell}C_k\prod_{n=1}^{k-1}(z-\lambda_n),\qquad z\in\C\setminus\R.
\end{equation*}
If $\ell=0$, the sum and products are understood as empty and equal to 0 and 1, respectively.

Moreover, $m_{\widehat T_{\ell}}$ can be understood as the $m$-function for the Friedrichs realization, $\widehat T_{\ell}$, of an $\ell$-times Darboux transformed $\widehat \tau_{\ell}$ so that $\sigma(\widehat T_{\ell})=\sigma(T)\cup\{\lambda_j\}_{j=1}^\ell$.
\end{proposition}
\begin{proof}
Follows from the previous discussion (after a possible Liouville transform) by using the naturally normalized system of solutions at the endpoint corresponding to $\ell$ on the left-hand side of \eqref{mFunction}. Note that the solutions will still be naturally normalized after the application of Darboux transforms by Corollary \ref{cor:cannorm}.
\end{proof}

\begin{remark}\label{Remarkm}
    Due to Remark \ref{RemThetaUnique}, a naturally normalized system $($at $x = a$$)$ is essentially unique up to the addition of $f(z) \varphi(z,x)$ to $\theta(z,x)$, where $f$ is an arbitrary real polynomial of degree $\ell_a$. Under such a transformation the $m$-function in \eqref{mFunction} would become $\overset{\circ}{m}_T = m_T - f$. In particular, we can always normalize $\theta$ such that the new $m$-function has the simpler form
    \begin{align} \label{mTilde}
        \overset{\circ}{m}_T(z) = \Bigg[\prod_{j=1}^{\ell_a}(z-\lambda_j)\Bigg]m_{\widehat T_{\ell}}(z).
    \end{align}
\end{remark}

Notice that one implication of Proposition \ref{Prop:Nevanlinna1} is that the $m$-function satisfying \eqref{mFunction} for any naturally normalized system is not a Nevanlinna--Herglotz function unless $\ell_a=0$. In fact, as a corollary we see that under the assumptions of the previous theorem, the $m$-function of the original problem is a type of generalized Nevanlinna--Herglotz function.

\begin{corollary}\label{CorKappa}
In addition to the assumptions and notation of Proposition \ref{Prop:Nevanlinna1}, let $\kappa=\lfloor(\ell+1)/2\rfloor.$
Then any singular $m$-function $m_T$ coming from \eqref{mFunction} with $\varphi$, $\theta$ are naturally normalized is in the subclass $N_\kappa^\infty$ consisting of generalized Nevanlinna--Herglotz functions with $\kappa$ negative squares, no nonreal poles, and the only generalized pole of nonpositive type at infinity.
\end{corollary}
\begin{proof}
We first point out that $N_0=N_0^\infty$ is the class of Nevanlinna--Herglotz functions. Also, by Remark \ref{Remarkm} the $m$-function $m_T$ can be written as $\overset{\circ}{m}_T + f$, with $\overset{\circ}{m}_T$ of the form \eqref{mTilde}, and $f$ a real polynomial of degree $\ell$. Furthermore, the $m$-function $m_{\widehat T_{\ell_a}}$ in \eqref{hatmFunction1} is a Nevanlinna--Herglotz function satisfying
\begin{equation}\label{mfunctionprop}
\lim_{y\to\infty}\frac{m_{\widehat T_{\ell}}(iy)}{iy}=0,
\end{equation}
as this holds for any Weyl $m$-function for Sturm--Liouville operators with one limit circle endpoint.  In particular, it has a Nevanlinna--Herglotz representation (see \cite[Eq. (6.2.39)]{GNZ23}).

Moreover, by performing one more Darboux transformation with nonprincipal seed function, one would multiply $m_{\widehat T_{\ell}}$ by another simple pole and still have a Nevanlinna--Herglotz function (see \ref{hatmFunction2}). Thus, $m_{\widehat T_{\ell}}$ cannot decay and must grow sublinearly by \eqref{mfunctionprop}. In fact, it must grow since if it had a finite limit at infinity, one could redefine $\widehat \theta_{a, \ell_a}$ in \eqref{hatmfunctioneq} by adding this limit times $\widehat \varphi_{a,\ell_a}$ to reach a contradiction since the new $m$-function would decay.
We suspect this is well-known behavior for the Friedrichs $m$-function of this form, but could not find an explicit statement in the literature.

The generalized Nevanlinna--Herglotz property for $\overset{\circ}{m}_T$ of the form \eqref{mTilde} now follows from repeated applications of \cite[Cor.~3.6]{KST_MN} taking $\lambda=\lambda_j=\inf \sigma(H)$ in \cite[Eq.~(3.14)]{KST_MN} at each step and understanding the infinite limit in \cite[Cor. 3.6]{KST_MN} as unbounded growth of the ratio. In particular, the first and then every other multiplication by $(z- \lambda_j)$ raises the index $\kappa$ in $N_\kappa^\infty$ by one. As the class $N_\kappa^\infty$ is invariant under addition of a real polynomial $f$ of degree $\ell \leq 2\kappa$ (see \cite[p.~190]{FL10}), the claim follows for the general $m_T$. 
\end{proof}

A few remarks are now in order.
\begin{remark}\label{remarkregularization}
$(i)$ Proposition \ref{Prop:Nevanlinna1} and Corollary \ref{CorKappa} remain true by assuming Hypothesis \ref{Hypothesis} holds with a finite regularization index at only one endpoint, and the other endpoint is nonoscillatory, with no further restrictions.\\[1mm]
$(ii)$ We remark that since the seed function was chosen to be nonprincipal at both endpoints at each step, this is the optimal method to regularize the endpoint with smallest regularization index in the sense of the least number of transforms and adding the fewest eigenvalues by Theorem \ref{Thm:Transformation}.\\[1mm]
$(iii)$ The definition of the regularization index $\ell_a$ at $x = a$ is closely related to the angular momentum $l \geq -\frac{1}{2}$ of a perturbed Bessel operator studied in the series of papers \cite{KST_Inv}--\cite{KT13}. More precisely,  $\ell_0 = \floor{l + \frac{1}{2}}$, that is,~in our language the perturbed Bessel operator studied in \cite{KST_Inv}--\cite{KT13} has a regularization index $\floor{l + \frac{1}{2}}$ at $x=0$. 
\end{remark}

So far we applied Darboux transforms until one of the endpoints (in our case $a$) is in the limit circle case. If the regularization indices are equal, then the Darboux transformed expression $\widehat\tau_{\ell_a}$ is in the limit circle nonoscillatory case at both endpoints. Otherwise, assuming $\ell_a<\ell_b$, we now want to continue the procedure $\ell_b-\ell_a$-times choosing seed functions with nonprincipal behavior at both endpoints as before. Care will now be needed as the principal/nonprincipal behavior of the naturally normalized system at the endpoint $x=a$ will swap with every further transform by Proposition \ref{PropPhiTheta}.

Assume now that $0\leq\ell_a<\ell_b<\infty$, and that we have completed $\ell_a$ Darboux transforms to arrive at \eqref{hatmFunction1}. As $\widehat\varphi_{a,\ell_a}$ and $\widehat\varphi_{b,\ell_a}$ are still principal at $a$ and $b$, respectively, the next Darboux transformation is exactly the same as previously. Hence, preceding as before, choosing $\lambda_{\ell_a+1}<\lambda_{\ell_a}$, yields
\begin{align}\label{hatmFunction2}
    \widehat \theta_{a,\ell_a+1}(z,x) +m_{\widehat T_{\ell_a+1}}(z) \widehat \varphi_{a,\ell_a+1}(z,x) = \prod_{j=1}^{\ell_a+1}(z-\lambda_j)^{-1} D(z) \widehat \varphi_{b,\ell_a+1}(z,x), \notag \\
     x \in (a,b), \quad z \in \C \setminus \R,
\end{align}
where
\begin{equation*}
m_{\widehat T_{\ell_a+1}}(z)=\Bigg[\prod_{j=1}^{\ell_a+1}(z-\lambda_j)^{-1}\Bigg]\left[m_{T}(z)-\sum_{k=1}^{\ell_a+1}C_k\prod_{n=1}^{k-1}(z-\lambda_n)\right],\qquad z\in\C\setminus\R.
\end{equation*}
However, $\widehat\varphi_{a,\ell_a+1}$ will now be nonprincipal at $a$, while $\widehat\theta_{a,\ell_a+1}$ will be principal at $a$, by Proposition \ref{PropPhiTheta}. Note that $\widehat\varphi_{b,\ell_a+1}$ will remain principal at $b$ (which will continue to be the case). This time $m_{\widehat T_{\ell_a+1}}$ is the $m$-function of the self-adjoint realization $\widehat T_{\ell_a+1}$ of $\widehat \tau_{\ell_a+1}$ with the Neumann-type boundary condition at $x =a$ defined via $\widehat\varphi_{a,\ell_a+1}$ and either Dirichlet at $x=b$ if $\ell_b=\ell_a+1$ or no boundary conditions at $x=b$ if $\ell_b>\ell_a+1$. While this step once again adds an eigenvalue, we will have to modify the next seed function chosen since the principal/nonprincipal behavior near $x=a$ interchanged.

Assuming now that $\ell_b>\ell_a+1$ (so that at least one more step is needed in the regularization process), we choose the seed function with $\lambda_{\ell_a+2}<\lambda_{\ell_a+1}$,
\begin{equation*}
\psi_{\ell_a+2}(x)=\widehat\varphi_{a,\ell_a+1}(\lambda_{\ell_a+2},x),
\end{equation*}
which is nonprincipal at both endpoints (as the nonprincipality at $x=b$ remained unchanged through each transformation). The main difference now is that no constant shift is needed for the choice of seed function above. Therefore we need only apply $A_{\psi_{\ell_a+2}}$ to \eqref{hatmFunction2} on this step to arrive at
\begin{align}\label{hatmfunctioneq2}
    \widehat \theta_{a,\ell_a+2}(z,x) +m_{\widehat T_{\ell_a+2}}(z) \widehat \varphi_{a,\ell_a+2}(z,x) = \prod_{j=1}^{\ell_a+1}\frac{1}{z-\lambda_j} D(z) \widehat \varphi_{b,\ell_a+1}(z,x), \notag \\
     x \in (a,b), \quad z \in \C \setminus \R,
\end{align}
where
\begin{align}\label{hatmFunction4}
m_{\widehat T_{\ell_a+2}}(z)=(z-\lambda_{\ell_a+2})\Bigg[\prod_{j=1}^{\ell_a+1}(z-\lambda_j)^{-1}\Bigg] \left[m_{T}(z)-\sum_{k=1}^{\ell_a+1}C_k\prod_{n=1}^{k-1}(z-\lambda_n)\right],\notag\\
z\in\C\setminus\R.
\end{align}
In particular, \eqref{hatmFunction4} shows that we did not add an eigenvalue in this step and instead added a zero to the $m$-function (since $\lambda_{\ell_a+2}$ was not an eigenvalue of the previous step). Moreover, the principal/nonprincipal behavior at $x=a$ once again swaps, so $m_{\widehat T_{\ell_a+2}}$ is the $m$-function for the corresponding Friedrichs realization once again (with Dirichlet boundary conditions at $x=b$ if and only if $\ell_b=\ell_a+2$).

In case $\ell_b>\ell_a+2$ we have to apply another Darboux transform (from then on the pattern will repeat). Similarly to before, as the endpoint $x = a$ is in the limit circle case, $\widehat \theta_{a, \ell_a+2}$ defines a Neumann-type boundary condition at this endpoint. In fact, $\lambda_{\ell_a+2}$ is the lowest eigenvalue of the self-adjoint realization of $\widehat \tau_{\ell_a+2}$ with the $\widehat \theta_{a, \ell_a+2}$-boundary condition at $x = a$. This mean that we can apply \cite[Cor.~2.4]{GST} to conclude that for any $\lambda_{\ell_a + 3} < \lambda_{\ell_a+2}$ we have that 
\begin{align*}
   \psi_{\ell_a+3}(x)=\widehat\theta_{a,\ell_a+2}(\lambda_{\ell_a+3},x). 
\end{align*}
is nonvanishing and nonprincipal at the endpoint $x = b$ (as otherwise $\lambda_{\ell_a+3}$ would be a smaller eigenvalue). In particular, no constant shift in the $m$-function is required. We can apply $\frac{1}{z-\lambda_{\ell_a+3}}A_{\psi_{\ell_a+3}}$ to both sides of \eqref{hatmfunctioneq2}
which adds an eigenvalue and swaps the principal/nonprincipal behavior at $x=a$.

Finally, this process can be iterated $\ell_b-\ell_a$ times to arrive at the case of both regularization indices becoming zero and the $m$-function
\begin{align}\label{hatmFunction5}
m_{\widehat T_{\ell_b}}(z)&=\Bigg[\prod_{i=\ell_a+1}^{\ell_b}(z-\lambda_i)^{(-1)^{i-\ell_a}}\Bigg]\Bigg[\prod_{j=1}^{\ell_a}(z-\lambda_j)^{-1}\Bigg] \notag\\
&\qquad\times\left[m_{T}(z)-\sum_{k=1}^{\ell_a+1}C_k\prod_{n=1}^{k-1}(z-\lambda_n)\right],\qquad z \in \C \setminus \R.
\end{align}
We remark that the additional poles and zeros added after the $\ell_a$th step in \eqref{hatmFunction5} will necessarily interlace by construction, and a total of $\lfloor (\ell_a+\ell_b+1)/2\rfloor$ eigenvalues were added during the regularization process.

We now summarize this regularization process by extending Proposition \ref{Prop:Nevanlinna1}:
\begin{proposition}\label{Prop:Nevanlinna2}
Assume Hypotheses \ref{HypoInt} and \ref{HypoDarb}. Furthermore, assume Hypothesis \ref{Hypothesis} is satisfied at $x=a, b$ with $\ell_a,\ell_b<\infty$, let $\ell=\min\{\ell_a,\ell_b\}$, $N=\max\{\ell_a,\ell_b\}$, and denote by $T$ either the unique self-adjoint realization of $\tau$ if $\ell_a,\ell_b\neq0$ or the Friedrichs realization otherwise. Then for each choice of $N$ real numbers satisfying $\lambda_N<\dots<\lambda_{1}<\inf\sigma(T)$ there is a Nevanlinna--Herglotz function $m_{\widehat T_N}$ and $\ell+1$ constants $C_j=C_j(\lambda_j)\in\R$, $j = 1, \dots, \ell+1$ such that the Weyl $m$-function $m_T$ in \eqref{mFunction} can be written as
\begin{align}\label{Eq:mT}
m_{T}(z)=\Bigg[\prod_{i=\ell+1}^{N}(z-\lambda_i)^{(-1)^{i-\ell+1}}\Bigg]\Bigg[\prod_{j=1}^{\ell}(z-\lambda_j)\Bigg]m_{\widehat T_{N}}(z)+\sum_{k=1}^{\ell+1}C_k\prod_{n=1}^{k-1}(z-\lambda_n), \notag\\
z \in \C \setminus \R.
\end{align}
Moreover, if $|\ell_a-\ell_b|$ is even, then $m_{\widehat T_{N}}$ can be understood as the $m$-function for the Friedrichs realization, $\widehat T_{N}$, of a $N$-times Darboux transformed quasi-regular $\widehat \tau_{N}$. If $|\ell_a-\ell_b|$ is odd, then $m_{\widehat T_{N}}$ can be understood as the $m$-function for a self-adjoint realization with Dirichlet boundary condition at the endpoint with larger index, and a Neumann-type boundary condition at the other endpoint.

In both cases, $\sigma(\widehat T_{N})=\sigma(T)\cup\{\lambda_j\}_{j=1}^\ell \cup\{\lambda_{n}\}_{n\in S}$ where $S=\{\ell+1,\ell+3,\dots, \ell+1+2\lfloor (N-\ell-1)/2\rfloor\}$ so that the number of eigenvalues added during the regularization process is $\lfloor (\ell_a+\ell_b+1)/2\rfloor$.
\end{proposition}
\begin{proof}
Follows from the previous discussion combined with Prop.~\ref{Prop:Nevanlinna1}.
\end{proof}

As both endpoints of $\widehat\tau_N$ are now in the limit circle nonoscillatory case, a corresponding regular expression can be found by utilizing \cite[Thm. 8.3.1]{Ze05}, effectively regularizing $\tau$. Once again, since the seed function was chosen to be nonprincipal at both endpoints at each step, this is the optimal method to regularize both endpoints in the sense of the least number of transforms and adding the fewest eigenvalues by Corollary \ref{CorDarb}. We also point out that the analog of Remark \ref{Remarkm} is true in this more general case as well, that is, the sum in \eqref{Eq:mT} can be removed.

We end by posing the following question:

\begin{problem}  
If both $\ell_a,\ell_b$ are infinite $($or one is infinite and the other is not defined but the problem is still nonoscillatory for all $\lambda\in\R$$)$, is the $m$-function corresponding to naturally normalized systems in some larger special class of functions than generalized Nevanlinna--Herglotz functions?
\end{problem}

Given the relation $\kappa = \floor{(\ell+1)/2}$, we would expect the $m$-functions when $\ell = \infty$ to display superpolynomial growth. Moreover, the corresponding special class of functions would have to be invariant under the addition of arbitrary entire functions which are real on the real line by Remark \ref{RemThetaUnique}.
\section{Examples}\label{sectexamples}

We now turn to a few examples  for which we can determine the regularization indices and, in some cases, write down naturally normalized systems explicitly. We begin by working out the generalized Bessel equation in full detail, illustrating our previous results. We then discuss the Jacobi differential equation, a Mie-type potential on a finite interval, power potentials on $(0,\infty)$, and end with the Laguerre differential equation.

\subsection{Generalized Bessel equation}\label{subBessel}
We start by recalling the generalized Bessel equation following the analysis in \cite{GNS21} (see also \cite{FGKLNS21}, \cite{GLNPS21}, and \cite[Sect. 8.4]{GNZ23}).
Let $a=0$, $b\in(0,\infty)$, and consider
\begin{align*}
\begin{split}
\tau_{\d,\nu,\g} = x^{-\d}\left[-\frac{d}{dx}x^\nu\frac{d}{dx} +\frac{(2+\d-\nu)^2\g^2-(1-\nu)^2}{4}x^{\nu-2}\right],\\
2+\d-\nu>0,\quad \g\geq0,\quad x\in(0,b),  
\end{split}
\end{align*}
which is possibly singular at the endpoint $x=0$ (depending on parameter choices) and always regular at $x=b$ when $b\in(0,\infty)$. (As we are concerned with the endpoint $x=0$ here, the case $b=\infty$ can be treated similarly by simply replacing $b$ with some $c\in(0,\infty)$ throughout.) We recall some basic facts about this equation and the interested reader is directed to \cite[Sect. 8.4]{GNZ23} for more details.

This problem is nonoscillatory at $x = 0$ (and $x = b$) for all $\lambda\in \R$ under the parameter choices above, and principal and nonprincipal solutions at $x=0$ with $\lambda = 0$ are given by
\begin{align*}
u_{0;\d,\nu,\g}(0,x)&=x^{[1-\nu+\g(2+\d-\nu)]/2},\quad \gamma \in [0,\infty), \\
\np_{0;\d,\nu,\g}(0,x)&=\begin{cases}
\displaystyle \dfrac{1}{\g (2+\d-\nu)}x^{[1-\nu-\g(2+\d-\nu)]/2},& \g\in(0,\infty),\\[3mm]
\displaystyle \ln(1/x)x^{(1-\nu)/2},& \g=0,
\end{cases} \\
&\hspace{2.8cm} 2+\d-\nu>0,\quad x\in(0,b).
\end{align*}
The nonprincipal solution behavior shows that $\tau_{\d,\nu,\g}$ is in the limit circle case at $x=0$ if $0\leq\g<1$ and in the limit point case at $x=0$ when $\g\geq1$.
Furthermore, one readily verifies that Hypothesis \ref{Hypothesis} holds at $x=0$ for this example as multiplying $u$, $\np$, and $r$, and integrating, yields a multiple of $x^{2+\d-\nu}$ (times a log term when $\g=0$), which has positive power under the parameter assumptions.

Next, one can construct $\varphi_{n;\d,\nu,\g}(0,x)$ by choosing $\varphi_{0;\d,\nu,\g}(0,x)=u_{0;\d,\nu,\g}(0,x)$ and then iterating the recursion given in \eqref{phin}. For instance, one easily finds
\begin{align*}
\varphi_{1;\d,\nu,\g}(0,x)&=-\frac{1}{(1+\g)(2+\d-\nu)^2}x^{2+\d-\nu+([1-\nu+\g(2+\d-\nu)]/2)}\\
&=-\frac{1}{(1+\g)(2+\d-\nu)^2}x^{2+\d-\nu}\varphi_{0;\d,\nu,\g}(0,x).
\end{align*}
Similarly, to construct $\theta_{n;\d,\nu,\g}(0,x)$, one can first choose $\theta_{0;\d,\nu,\g}(0,x)=\np_{0;\d,\nu,\g}(0,x)$ and then iterate \eqref{thetaNformula} while choosing $A_n=-C_n$, where $C_n$ is defined implicitly by
\begin{align*}
    \int_x^c\theta_0(t)\theta_{n-1}(t)r(t)dt = D_n x^{(2+\d-\nu)n+([1-\nu+\g(2+\d-\nu)]/2)} + C_n.
\end{align*}
In general, iterating the recursions yields a pattern that can then be proven by induction. In particular, by denoting with $H_{k}$ the $k$-th harmonic number,
\begin{equation*}
H_0=0,\quad H_k=\sum_{j=1}^k\dfrac{1}{j},
\end{equation*}
one now finds for $2+\d-\nu>0,\; \gamma \in [0,\infty),\ x\in(0,b)$,
\begin{align*}
\varphi_{n;\d,\nu,\g}(0,x)&=\dfrac{(-1)^n\Gamma(1+\g)}{(2+\d-\nu)^{2n}n!\Gamma(n+1+\gamma)}x^{(2+\d-\nu)n+([1-\nu+\g(2+\d-\nu)]/2)}, \\
\theta_{n;\d,\nu,\g}(0,x)&=\begin{cases}
\displaystyle \dfrac{(-1)^n \Gamma(1-\g)}{\g (2+\d-\nu)^{2n+1}n!\Gamma(n+1-\g)}x^{(2+\d-\nu)n+([1-\nu-\g(2+\d-\nu)]/2)},\\[2mm]
\hspace*{7.2cm}\g\in(0,\infty)\backslash\N,\\[3mm]
\displaystyle \dfrac{(-1)^n [2H_n-(2+\d-\nu)\ln(x)]}{ (2+\d-\nu)^{2n+1}(n!)^2}x^{(2+\d-\nu)n+[(1-\nu)/2]},\\[2mm]
\hspace*{8.4cm}\g=0.
\end{cases}
\end{align*}
We point out that for brevity we have not included the remaining logarithmic $\theta_{n;\d,\nu,k}$ terms that occur when considering $\g=k\in\N$. The case $\gamma=0$ is included to illustrate the main difference in these cases.
These expressions allow one to readily find that $\ell_0=\floor{\gamma}$ and $\theta_{n;\d,\nu,\g}\in L^2((0,b);x^\d dx)$ for $n>(\gamma-1)/2$, $\gamma\in[0,\infty)\backslash\N$. Moreover, $\theta_{n;\d,\nu,\g}$ was constructed to satisfy the normalization \eqref{normTheta} for $n > \ell_0$, implying that the resulting $\theta_{\delta, \nu, \gamma}(z,x)$ will be naturally normalized.

Finally, one concludes that the following expressions hold for the solutions $\varphi,\theta$:
\begin{align*}
\varphi_{\d,\nu,\g}(z,x)&=\sum_{n=0}^\infty \dfrac{(-1)^n x^{(2+\d-\nu)n+([1-\nu+\g(2+\d-\nu)]/2)}\Gamma(1+\g)}{(2+\d-\nu)^{2n}n!\Gamma(n+1+\gamma)}z^n,\quad \g\in[0,\infty), \\
\theta_{\d,\nu,\g}(z,x)&=\begin{cases}
\displaystyle \sum_{n=0}^\infty \dfrac{(-1)^n  x^{(2+\d-\nu)n+([1-\nu-\g(2+\d-\nu)]/2)}\Gamma(1-\g)}{\g (2+\d-\nu)^{2n+1}n!\Gamma(n+1-\g)}z^n,\\[2mm]
\hspace*{6.85cm}\g\in(0,\infty)\backslash\N,\\[2mm]
\displaystyle \sum_{n=0}^\infty \dfrac{(-1)^n [2H_n-(2+\d-\nu)\ln(x)]x^{(2+\d-\nu)n+[(1-\nu)/2]}}{ (2+\d-\nu)^{2n+1}(n!)^2}z^n,\\[2mm]
\hspace*{8.4cm}\g=0,
\end{cases} \\
&\hspace{4.7cm} 2+\d-\nu>0,\quad x \in (0,b),\quad z\in\C.
\end{align*}
Up to multiples, these solutions are identical with the ones given in \cite[Sect.~6]{GNS21} constructed out of the Bessel functions $J_{\pm\gamma}$ (and $Y_0$ for $\gamma = 0$). This fact can be seen from the series expansions of $J_{\pm\gamma}$, $Y_0$ around $0$ (see \cite[Eqs.~10.2.2, 10.8.2]{DLMF}).

\medskip

For the next few examples, we focus on the leading behavior of the principal and nonprincipal solutions only to illustrate our main theory.

\subsection{Jacobi equation}\label{subJac}

This example considers the Jacobi differential equation. See \cite{GLPS23} for more details and an extensive list of references (see also \cite[Sects. 9 and 23]{Ev05}). In particular, we consider the Jacobi differential expression 
\begin{align*}
\begin{split} 
\tau_{\a,\b} = - (1-x)^{-\a} (1+x)^{-\b}(d/dx) \big((1-x)^{\a+1}(1+x)^{\b+1}\big) (d/dx),&     \\ 
\a, \b \in \R, \quad x \in (-1,1).
\end{split} 
\end{align*}
This example follows from Remark \ref{remarkexamples} $(ii)$ by setting $\nu_{-1}=\b+1,\ \d_{-1}=\b$ and $\nu_1=\a+1,\ \d_1=\a$ for the endpoints $x=\pm1$ to arrive at $\ell_{-1} = \floor{|\beta|}$ and $\ell_{+1} = \floor{|\alpha|}$. 
Therefore, we recover the well-known fact that this equation is in the limit circle case at both endpoints if and only if $\a,\b\in(-1,1)$.

It now follows that $\tau_{\alpha, \beta}$ can be transformed to a Sturm--Liouville differential expression which is in the limit circle case at both endpoints via a sequence of $\max\{ \floor{|\alpha|}, \floor{|\beta|} \}$ Darboux transforms (and no shorter sequence will achieve this). Explicit examples of Darboux transforms of the Jacobi equation appear naturally in the study of exceptional Jacobi polynomials (see \cite{G-UKM}, \cite[Sect.~5]{G-UMM}; for the case of Darboux--Crum transformations see \cite{B19}). 

\subsection{Mie-type (incl.~inverse quartic) potentials on a finite interval}\label{MiePotential} 
We will now consider an example of a Schr\"odinger equation on a finite interval with $\ell_0 = \infty$ and Weyl asymptotics. Consider the Schr\"odinger differential expression given by
\begin{align*}
    \tau_\mu = -\frac{d^2}{dx^2} + \frac{\mu^2}{x^{2\mu+2}}+\frac{\mu(1-\mu)}{x^{\mu+2}}, \qquad  \mu > 0, \quad x \in (0,b).
\end{align*}
Note that in the special case $\mu = 1$ we recover the inverse quartic potential studied in \cite{LN16}. One can compute that $\varphi_{0;\mu}(0,x) = x\exp(-x^{-\mu})$ solves $\tau y = 0$.
A linearly independent solution is given by $\theta_{0;\mu}(0,x) = x \exp(-x^{-\mu}) \int_x^ct^{-2}\exp(2t^{-\mu}) dt$, which is nonprincipal. Moreover, $W(\theta_{0;\mu}(0, \, \cdot \,), \varphi_{0;\mu}(0, \, \cdot \,)) = 1$ holds.  Note that
\begin{align*}
    \varphi_{0;\mu}(0,x)\theta_{0;\mu}(0,x) = \int_x^c \underbrace{\frac{x^2}{t^2}\exp\big(2[t^{-\mu} - x^{-\mu}]}_{\leq 1}\big) dt,
\end{align*}
which is bounded, hence integrable, at $x = 0$. Thus Hypothesis \ref{Hypothesis} is satisfied.

Let $\theta_{n;\mu}(0,x)$ be given through the recursion \eqref{thetaNformula} with $B_n = 0$ and arbitrary $A_n \in \R$. 
One can now show inductively that
\begin{align*}
|\theta_{n;\mu}(0,x)| \gtrsim c_n |\theta_{0;\mu}(0,x) x^{n(\mu+2)}|, \quad x \to 0^+,
\end{align*}
for some sequence $c_n > 0$, cf. \cite[Lem.~8]{LN16}. As $\theta_{0;\mu}(0,x) \propto x^\mu \exp(x^{-\mu})$ diverges exponentially for $x \to 0^+$, so does $\theta_{n;\mu}(0, x)$ for all $n \geq 0$, and we conclude that $\ell_0 = \infty$. At the same time, it follows from Lemma \ref{LemmaWeyl} that the eigenvalues $\lambda_n$ of any self-adjoint realization of $\tau_\mu$ satisfy Weyl asymptotics.

\subsection{Power function potentials on the half-line}\label{subPoly}
The equation considered here is an example on an infinite interval with $\ell_\infty=\infty$ that has eigenvalues satisfying the growth $n^\g$ for any $\g\in(1,2)$ (i.e., trace class resolvent but slower growth than Weyl given in \eqref{eq:Weyl}). Consider the half-line Schr\"odinger differential expression
\begin{equation*}
\tau_\alpha=-\frac{d^2}{dx^2}+x^\alpha,\qquad \alpha>0,\quad x\in(0,\infty).
\end{equation*}
Note that the problem $\tau_\alpha y=\lambda y$ is limit point and nonoscillatory for all $\lambda\in\R$ at $x=\infty$. Specializing to the case $\lambda=0$ here for simplicity, linearly independent solutions to $\tau_\alpha y=0$ are given by (see \cite[Eqs.~10.40.1 and 10.40.2]{DLMF})
\begin{align*}
\varphi_{0;\alpha}(0,x)&=x^{1/2} K_{\frac{1}{\alpha+2}}\big(\big[2x^{(2+\alpha)/2}\big]/(\alpha+2)\big) \underset{x\to\infty}{\propto} x^{-\alpha/4} e^{-\frac{2}{2+\alpha}x^{(2+\alpha)/2}}, \\
\theta_{0;\alpha}(0,x)&=-(2/(\alpha+2))x^{1/2} I_{\frac{1}{\alpha+2}}\big(\big[2x^{(2+\alpha)/2}\big]/(\alpha+2)\big) \underset{x\to\infty}{\propto} x^{-\alpha/4} e^{\frac{2}{2+\alpha} x^{(2+\alpha)/2}},
\end{align*}
where $I_\nu, K_\nu$ are the typical modified Bessel functions (see \cite[Sect. 10.25]{DLMF}), which allows one to directly verify that $x=\infty$ is in the limit point case as only the first solution is square-integrable near $\infty$. One can verify that $W(\theta_{0;\alpha}(0, \, \cdot \,), \varphi_{0;\alpha}(0, \, \cdot \,)) = 1$ by applying the Wronskian given in \cite[Eq.~10.28.2]{DLMF}.

The asymptotic behavior of the solutions shows that Hypothesis \ref{Hypothesis} holds if and only if $\alpha>2$. Furthermore, when iteratively constructing $\theta_\alpha(z,x)$ using \eqref{thetaNformula}, the terms $\theta_{n;\alpha}(0,x)$ will always include an exponentially growing term for $x\to\infty$. This allows one to immediately conclude that $\ell_\infty=\infty$ for $\alpha>2$.

We end this example by remarking that the eigenvalue asymptotics for the problem $\tau_\alpha y=zy$ are explicitly given by (see \cite[Eq. (7.1.7)]{Ti62})
\begin{equation*}
\lambda_n\underset{n\to\infty}{\sim} \bigg[\frac{2\pi^{1/2} \alpha \Gamma(\frac{3}{2}+\frac{1}{\alpha})}{\Gamma(\frac{1}{\alpha})}\bigg]^{2\alpha/(\alpha+2)} n^{2\alpha/(\alpha+2)},\qquad \alpha>0,
\end{equation*}
which satisfies growth $n^\g$ for any $\g\in(1,2)$ under the additional assumption $\alpha>2$ required above.

\subsection{Laguerre equation}\label{subLag} 
For our final example, we consider the Laguerre equation which serves more as a non-example in the sense that Hypotheses \ref{hypothesisTrace} and \ref{Hypothesis} are not satisfied and eigenvalues of this equation grow like $n$ (see \cite[Sects.~10 and 27]{Ev05}, \cite[Sect.~6]{GLN20}, \cite[App.~D]{GLPS23}, and \cite[Ch.~13]{DLMF}).
The form of the Laguerre differential expression we will study is given by 
\begin{equation*}
\tau_\a = -x^{1-\a}e^x \frac{d}{dx}x^{\a}e^{-x} \frac{d}{dx}, \qquad  \a \in (0,\infty)\backslash\N,\quad x \in (0,\infty).
\end{equation*}
We will be interested in the endpoint $x = \infty$ as $\tau_\alpha|_{(c, \infty)}$ does not have trace class resolvents, meaning that Hypothesis \ref{hypothesisTrace} (and hence Hypothesis \ref{Hypothesis}) is not satisfied at this endpoint. Again, for simplicity we have restricted the set of admissible $\alpha$, however the general case $\alpha \in \R$ can be handled in a similar fashion.

Linearly independent solutions of $\tau_\a y=\lambda y$ with $\lambda\in\R$ are then given by (see \cite[Eqs.~13.2.6 and 13.2.26]{DLMF})
\begin{align*}
U(-\lambda,\a;x) \underset{x\to\infty}{\propto} x^\lambda, \quad
x^{1-\a} {}_1F_1(1-\a-\lambda, 2-\a; x)\underset{x\to\infty}{\propto} x^{-\lambda-\a}e^x ,\quad (\lambda+\a)\notin \N,
\end{align*}
where ${}_1F_1$ is the confluent hypergeometric function (also frequently denoted by $M$), $U$ the associated logarithmic solution, and the asymptotic behavior follows from \cite[Eqs.~13.7.1 and 13.7.3]{DLMF}. (For brevity, we once again remove some parameter choices that require different solutions to be chosen.) We immediately see the $\lambda$-dependence in the asymptotics as predicted by Theorem \ref{TFAE}. Also, we can directly verify that Hypothesis \ref{Hypothesis} does not hold at $\infty$ since multiplying the lead behaviors by $r(x)=x^{\a-1} e^{-x}$ gives $x^{-1}$, which is not integrable near $\infty$. The Wronskian of the two solutions can be computed via \cite[Eq.~13.2.36]{DLMF}.

Note that an entire system of solutions $\widetilde \varphi$, $\widetilde \theta$ exists by \cite{GZ06} as Hypothesis \ref{HypoGZ} is satisfied. In particular, one has for $\lambda\in\R$,
\begin{align}\label{LaguAsym}
\begin{split}
\widetilde \varphi(\lambda,x)&\underset{x\to\infty}{\propto} x^{\lambda}\in L^2((c,\infty); x^{\a-1} e^{-x} dx), \\
\widetilde \theta(\lambda,x)&\underset{x\to\infty}{\propto} x^{-\lambda-\a}e^x\notin L^2((c,\infty); x^{\a-1} e^{-x} dx),
\end{split}
\end{align}
where we dropped the $\alpha$-dependance for brevity. Again, the tilde indicates that $\widetilde \varphi$, $\widetilde \theta$ do not satisfy any normalization except being entire in $z$ and principal, respectively nonprincipal, at $x = \infty$. In fact, our standard normalization \eqref{phiNorm} cannot hold for this example due to the lead behavior differing depending on the choice of $z\in \C$ (as it has to by Theorem \ref{TFAE}) since for $z_1\neq z_2$ one has
\begin{equation*}
\frac{\widetilde \varphi(z_1,x)}{\widetilde \varphi(z_2,x)}\underset{x\to\infty}{\propto} x^{z_1-z_2},\qquad z_1,z_2\in\C,
\end{equation*}
and similarly for $\widetilde \theta$.

We can now expand $\widetilde \varphi$, $\widetilde \theta$ around $\lambda \in \R$: 
\begin{align}\label{BornLaguerre}
\widetilde \varphi(z,x) = \sum_{n \geq 0} \widetilde \varphi_n(\lambda, x)(z-\lambda)^n, \qquad \widetilde \theta(z,x) = \sum_{n \geq 0} \widetilde \theta_n(\lambda, x)(z-\lambda)^n.
\end{align}
By Lemma \ref{LongLem} we know that $\widetilde \varphi_n$, $\widetilde \theta_n$ will satisfy the recursion \eqref{y_n}. However, as Hypothesis \ref{Hypothesis} is not satisfied, we cannot use the recursions \eqref{phin} and \eqref{thetaNformula} anymore as some of the integrals will not converge. Instead we need to modify these recursions by replacing the integration limit $a$ by some $d \in (a,b)$ in order to find the behavior of the coefficients of the Born series for this example. This results in the following recursions: 
\begin{align}\nonumber
        \widetilde\varphi_{n}(x) &= A_{n} \widetilde\varphi_0(x) + B_n \widetilde\theta_0(x) 
        \\ \label{phinformula}
        &\quad + \widetilde\theta_0(x)\int_d^x \widetilde\varphi_0(t) \widetilde\varphi_{n-1}(t) r(t) dt + \widetilde\varphi_0(x) \int_x^c \widetilde\theta_0(t) \widetilde\varphi_{n-1}(t) r(t) dt,\\
        \widetilde\theta_{n}(x) &= A'_n \widetilde\varphi_0(x) + B'_n \widetilde\theta_0(x)  \nonumber
        \\\label{thetanformula}
        &\quad + \widetilde\theta_0(x)\int_{d'}^x \widetilde\varphi_0(t) \widetilde\theta_{n-1}(t) r(t) dt + \widetilde\varphi_0(x) \int_x^{c'} \widetilde\theta_0(t) \widetilde\theta_{n-1}(t) r(t) dt,\\
        &\quad c,d,c',d'\in(a,b),\quad A_n,B_n,A_n',B_n'\in\R. \nonumber
\end{align}
Care needs to be taken in constructing $\widetilde \varphi_n$ in this setting to ensure that $\widetilde \varphi$ is still principal. For instance, the first term containing an integral in \eqref{phinformula} will behave like a constant times $\widetilde \theta_0$ for $x \to \infty$, meaning that $B_n$ must be chosen to cancel this behavior as otherwise $\widetilde \varphi$ would in fact behave like $\widetilde \theta_0$, contradicting $\widetilde \varphi$ being principal. 

Iterating \eqref{phinformula} and \eqref{thetanformula} for this example with $\widetilde \varphi_{0}$ and $\widetilde \theta_{0}$ as above (with appropriately chosen $B_n$) yields
\begin{align*}
\widetilde \varphi_{n}(\lambda,x)&\underset{x\to\infty}{\propto} [\ln(x)]^n x^{\lambda}\in L^2((c,\infty); x^{\a-1} e^{-x} dx), \\
\widetilde \theta_{n}(\lambda,x)&\underset{x\to\infty}{\propto} [\ln(x)]^n x^{-\lambda-\a}e^x\notin L^2((c,\infty); x^{\a-1} e^{-x} dx),\qquad n\in\N_0,\quad \lambda\in\R.
\end{align*}
We see that in the absence of Hypothesis \ref{Hypothesis} the coefficients $\widetilde \varphi_{n}$ and $\widetilde \theta_{n}$ near $x=\infty$ can grow in $n$ as opposed to \eqref{PhiRhoEst} and \eqref{ThetaLimit}, respectively. This behavior is expected by Theorem \ref{TFAE} as higher order terms in the Born series \eqref{BornLaguerre} cannot be neglected and must contribute to the lead behavior near $x = \infty$. This is true in general whenever Hypothesis \ref{Hypothesis} (or equivalently Hypothesis \ref{hypothesisTrace}) is not satisfied, as explained in Remark \ref{RemarkBorn}.

\appendix
\section{Proofs of Lemma \ref{LongLem} and Theorem \ref{TFAE}}\label{AppendixA}

This appendix contains the more technical proofs not included in the main text.
We begin with a complete proof for Lemma \ref{LongLem}, which requires more technical arguments and will need the following standard result (cf.~\cite[Thm. 1.6.1]{Ze05}). 
\begin{lemma}\label{Banach}
    The following initial value problem,
    \begin{align*}
        \tau f = z f, \qquad f(z,x_0) =  y_0, \qquad f^{[1]}(z,x_0) = y_1,
    \end{align*}
    has for any $x_0 \in (a,b)$ and $y_0,y_1,z \in \mathbb C$ a unique solution $y \in AC_{loc}((a,b))$, $y^{[1]} \in AC_{loc}((a,b))$. Moreover, for any compact $K \subset (a,b)$, we have that the mapping  $z \mapsto y(z,x)|_{x\in K} \in L^\infty(K)$ is locally Lipschitz continuous.
\end{lemma}
\begin{proof}[Proof of Lemma \ref{LongLem}]
    $(i)$  First note by the assumption $a < c < d < b$ the differential expression $\tau|_{(c,d)}$ is regular at both endpoints. Consider some point $x_0 \in (c,d)$ and let $c(z,x)$, $s(z,x)$ be the entire system of solutions of $\tau f = z f$ satisfying 
    \begin{equation*}
s(z,x_0) = 0=c^{[1]}(z,x_0), \qquad  s^{[1]}(z,x_0) = 1=c(z,x_0).
\end{equation*}
By a standard uniqueness result for differential equations, we have
\begin{align*}
    y(z,x) = y(z,x_0)c(z,x) + y^{[1]}(z,x_0)s(z,x).
\end{align*}
As $y(z,\, \cdot \, )$, $s(z,\, \cdot \, )$ and $c(z,\, \cdot \,)$ are all holomorphic in $z$, it follows that $y^{[1]}(z,x_0)$ must be meromorphic in $z$. However, as $y^{[1]}\in AC_{loc}((a,b))$ for all $z \in U$, it is in fact holomorphic. In particular, $y(z,x_0)$, $y^{[1]}(z,x_0)$ are locally Lipschitz for $z \in U$. Hence, by the previous Lemma \ref{Banach}, the solution $y(z,\, \cdot \, )$ depends locally Lipschitz on $z \in U$ in the space $L^\infty((c,d); dx)$. 

We can now conclude that for any $z \in U$ there exists a Lipschitz constant $L_z > 0$ such that
\begin{align*}
    |y(z+h,x)-y(z,x)| \leq L_z|h|, \quad x \in (c,d) \, \text{ and } \, z+h \in U \, \text{ with } \, |h| < \varepsilon,
\end{align*}
for $\varepsilon > 0$ small enough. In particular, it follows by dominated convergence that
\begin{align*}
    \lim_{h \to 0} \frac{y(z+h,\, \cdot \,)-y(z,\, \cdot \,)}{h} = \partial_z y(z,\, \cdot \,) 
\end{align*}
in the space $L^2((c,d); r(x) dx)$ for any $z \in U$. This shows that the mapping $U \to L^2((c,d); r(x) dx)$, $z \mapsto y(z, \, \cdot \,)$ is an $L^2((c,d); r(x) dx)$-valued holomorphic mapping. Thus in the following, we can use the theory of Banach-valued holomorphic functions (see, e.g.,~\cite[Ch.~3]{HP57}). In particular, we have the series expansion \ref{yPowerSeries}, but with $y_n \in L^2((c,d); r(x) dx)$. To show that indeed $y_n \in \dom(T_{max}^{(c,d)})$ we choose a smooth contour $\gamma \subset U$ going once around the point $z_0 \in U$ in the counterclockwise direction. Then as $y(z, \, \cdot \,)$ and $T_{max}^{(c,d)} y(z, \, \cdot \,) = z y(z, \, \cdot \,)$ are Bochner integrable on $\gamma$,  we conclude by Hille's Theorem (see \cite[Thm. 3.7.12]{HP57}) that 
\begin{align*}
    y_n(z_0, \, \cdot \,) = \frac{1}{2\pi i} \oint_{\gamma} \frac{y(z, \, \cdot \,)}{(z-z_0)^{n+1}} \, dz \in \dom(T_{max}^{(c,d)})
\end{align*}
and for $n > 0$
\begin{equation*}
    (\tau-z_0)y_n(z_0, \, \cdot \,) = \frac{1}{2\pi i} \oint_{\gamma} \frac{(\tau-z_0)y(z, \, \cdot \,)}{(z-z_0)^{n+1}} \, dz 
    = \frac{1}{2\pi i} \oint_{\gamma} \frac{y(z, \, \cdot \,)}{(z-z_0)^{n}} \, dz = y_{n-1}(z_0, \, \cdot \,).
\end{equation*}
Here we used that $T_{max}^{(c,d)}$ is a closed operator as it is the adjoint of the minimal operator $T_{min}^{(c,d)}$. This shows \eqref{y_n} and finishes the proof of $(i)$.

$(ii)$ Let us define the truncated sums $y^{(M)}(z,x) = \sum_{n=0}^M y_n(z_0, x)(z-z_0)^n$. One then obtains for $M > 0$ and $z\in U$ chosen such that the series \eqref{yPowerSeries} converges in $L^2((c,d); r(x)dx)$ the formula 
\begin{align*}
    T_{max}^{(c,d)} y^{(M)}(z,\, \cdot \,) = z y^{(M-1)}(z, \, \cdot \,) + z_0 y_M(z_0, \, \cdot \,)(z-z_0)^M. 
\end{align*}
Note that $\lim_{M\to \infty} y^{(M)}(z, \, \cdot \,) = y(z, \, \cdot \,)$ and $\lim_{M\to \infty} y_M(z_0, \, \cdot \,)(z-z_0)^M = 0$ in $L^2((c,d); r(x)dx)$ by assumption. Again, as $T_{max}^{(c,d)}$ is a closed operator, it follows that $y(z, \, \cdot \,) \in \dom(T_{max}^{(c,d)})$ and $T_{max}^{(c,d)}y(z,\, \cdot \,) = \tau y(z,\, \cdot \,) = z y(z,\, \cdot \,)$. This finishes the proof of $(ii)$.
\end{proof}

Next we provide the proof of Theorem \ref{TFAE}.

\begin{proof}[Proof of Theorem \ref{TFAE}]
    $(i) \Rightarrow (ii)$: By Corollary \ref{CorNorm} we know that a principal entire solution $\varphi$ satisfying $\lim_{x \to a^+}\frac{\varphi(z_1,x)}{\varphi(z_2,x)} = 1$ exists. By the uniqueness of the principal solution up to real multiples, we know that there is an entire non-vanishing function $f$ such that $\widetilde \varphi(z,x) = f(z) \varphi(z,x)$. The claim follows.

    $(ii) \Leftrightarrow (iii)$: Note that as $\widetilde \varphi$ is principal, we must have (cf.~Cor.~\ref{CorPrinc})
    \begin{align*}
        \lim_{x \to a^+}\frac{\widetilde \varphi(z_1,x)}{\widetilde \theta(z_2,x)} = 0 \ \text{ for all} \ z_1, z_2 \in \mathbb R.
    \end{align*}
    Moreover, we have the formula
    \begin{align}\label{WronskiLimit}
        \Big[W(\widetilde \theta(z_2,t), \widetilde \varphi(z_1,t))\Big]^{t=c}_{t=x} = (z_2-z_1) \int_x^c \widetilde \theta(z_2, t) \widetilde \varphi(z_1, t) r(t) dt,
    \end{align}
    which implies that the limit $\displaystyle\lim_{x \to a^+} W(\widetilde \theta(z_2, x), \widetilde \varphi(z_1, x))$ exists in the extended real line $\mathbb R \cup \lbrace \pm \infty \rbrace$.
    Hence we can apply L'H\^opital's rule to arrive at 
    \begin{align}\label{tildePhi}
        \lim_{x \to a^+} \frac{\widetilde \varphi(z_1,x)}{\widetilde \varphi(z_2,x)} = \lim_{x \to a^+} \frac{\Bigg(\cfrac{\widetilde \varphi(z_1,x)}{\widetilde \theta(z_2,x)}\Bigg)'}{\Bigg(\cfrac{\widetilde \varphi(z_2,x)}{\widetilde \theta(z_2,x)}\Bigg)'} &= \lim_{x \to a^+} \frac{\cfrac{W(\widetilde \theta(z_2,x), \widetilde \varphi(z_1,x))}{p(x)}}{\cfrac{W(\widetilde \theta(z_2,x), \widetilde \varphi(z_2,x)) \equiv 1}{p(x)}} \notag
        \\
        &=\lim_{x \to a^+} W(\widetilde \theta(z_2,x), \widetilde \varphi(z_1,x)), 
    \end{align}
    showing the equivalence of $(ii)$ and $(iii)$.

    $(iii) \Leftrightarrow (iv)$: Similar as before we obtain
    \begin{align}\nonumber
        \lim_{x \to a^+} \frac{\widetilde \theta(z_1,x)}{\widetilde \theta(z_2,x)} = \lim_{x \to a^+} \frac{\Bigg(\cfrac{\widetilde \varphi(z_1,x)}{\widetilde \theta(z_2,x)}\Bigg)'}{\Bigg(\cfrac{\widetilde \varphi(z_1,x)}{\widetilde \theta(z_1,x)}\Bigg)'} &= \lim_{x \to a^+} \frac{\cfrac{W(\widetilde \theta(z_2,x), \widetilde \varphi(z_1,x))}{p(x)}}{\cfrac{W(\widetilde \theta(z_1,x), \widetilde \varphi(z_1,x)) \equiv 1}{p(x)}}
        \\\label{tildeTheta}
        &=\lim_{x \to a^+} W(\widetilde \theta(z_2,x), \widetilde \varphi(z_1,x)).
    \end{align}

    $(ii)+(iii)+(iv) \Rightarrow (v)$: Note that due to formula \eqref{WronskiLimit} and the assumption that the limit $\lim_{x \to a^+} W(\widetilde \theta(z_2,x), \widetilde \varphi(z_1,x))$ exists, together with $\widetilde \theta$, $\widetilde \varphi$ being nonoscillatory at the endpoint $x = a$, condition $(v)$ follows for $z_1 \not = z_2$. It then follows for $z_1 = z_2$ due to $(ii)$ and $(iv)$. 

    $(v) \Rightarrow (i)$: This implication is trivial. 
    
    We have thus shown that $(i)$--$(v)$ are all equivalent. 
    Clearly for some fixed $z_1, z_2 \in \R$ with $z_1  \not = z_2$, the statements of $(ii)$--$(iv)$ are again pairwise equivalent. But then together they imply via \eqref{WronskiLimit} that $(i)$ holds for $\lambda = z_1, z_2$.
    
    It remains to deal with $(vi)$.

    $(vi) \Rightarrow (i)$: Assuming Hypothesis \ref{hypothesisTrace} at $x = a$, one can write the trace of the resolvent of any self-adjoint realization $T$ of $\tau|_{(a,c)}$, $c\in(a,b),$ with separated boundary conditions as
    \begin{equation}\label{trace}
    \textrm{Tr}_{L^2((a,c);r(x)dx)}\big((T-zI)^{-1}\big)=\int_a^c G(z,x,x)r(x) dx,\qquad z\in\rho(T),
    \end{equation}
    where $G(z,x,x)$ is the diagonal Green's function.
    
    In the special case that $T$ is the Friedrichs extension, the diagonal of the Green's function takes on the form
    \begin{equation}\label{diagGreens}
    G(z,x,x)=\widetilde\varphi(z,x)[\widetilde\theta(z,x)+m(z)\widetilde\varphi(z,x)],\qquad z\in\rho(T),
    \end{equation}
    where, as usual, $\widetilde \varphi$, $\widetilde \theta$ are principal, respectively nonprincipal, entire solutions at $x = a$ satisfying $W(\widetilde \theta, \widetilde \varphi) \equiv 1$, and $m(z)$ is the Weyl $m$-function for $T$ (see Section \ref{SectWeylm}). As \eqref{trace} converges by assumption, \eqref{diagGreens} implies that $(i)$ holds for $\lambda\in\rho(T)\cap\R$.
    
    $(i) \Rightarrow (vi)$: Take any $\lambda \in \R$ such that $\tau f = \lambda f$ becomes nonoscillatory at $x = a$. Now choose solutions $u$, $\np$ of $\tau f = \lambda f$ on $(a,b)$, such that $u$ is principal and $\np$ is nonprincipal at $a$. Define $x_\varepsilon = a + \varepsilon < c$ with $\varepsilon > 0$ sufficiently small. Then $\tau|_{(x_\varepsilon, c)}$ is regular at both endpoints. Hence, for each $x_\varepsilon$ we can define a fundamental system $\phi_\varepsilon(z, x)$, $\vartheta_\varepsilon(z, x)$, $x \in (x_\varepsilon, c)$, entire in $z$, such that
    \begin{align*}
        \phi_\varepsilon(z, x_\varepsilon) &=u(x_\varepsilon) \qquad \phi_\varepsilon^{[1]}(z, x_\varepsilon) = u^{[1]}(x_\varepsilon)
        \\
        \vartheta_\varepsilon(z, x_\varepsilon) &=\np(x_\varepsilon) \qquad
        \vartheta_\varepsilon^{[1]}(z, x_\varepsilon) = \np^{[1]}(x_\varepsilon).
    \end{align*}
    for all $z \in \C$. In other words, $\phi_\varepsilon$, $\vartheta_\varepsilon$ satisfy the boundary condition defined by $u$, $\np$ at the point $x_\varepsilon \in (a,c)$. Let $T_{\varepsilon}$ be the self-adjoint realization of $\tau|_{(x_\varepsilon, c)}$ satisfying the $\phi_\varepsilon$-boundary condition at $x=x_\varepsilon$ and the Dirichlet boundary condition at $x= c$. Associated to $T_{\varepsilon}$ let $m_\varepsilon$ be its Weyl $m$-function. Then  
    \begin{align*}
        \vartheta_\varepsilon(z,x) + m_\varepsilon(z) \phi_\varepsilon(z,x) 
    \end{align*}
    will satisfy the Dirichlet boundary condition at $x = c$ for all $z \in \C \setminus \sigma(T_{\varepsilon})$. 

    Note that $\phi_\varepsilon(\lambda, x) = u(x)$ and $\vartheta_\varepsilon(\lambda, x) = \np(x)$. If necessary, we can perturb the endpoint $c$ such that $u$ does not vanish at $x = c$. Then it follows that $\np(x) + m_\varepsilon(\lambda) u(x)$ satisfies the Dirichlet boundary condition at $x = c$, hence $m_\varepsilon(\lambda) = m \in \R$ is independent of $\varepsilon$. In particular we can now write
    \begin{align*}
    \textrm{Tr}_{L^2((x_\varepsilon,c);r(x)dx)}\big((T_{\varepsilon}-\lambda I)^{-1}\big)&=\int_{x_\varepsilon}^c u(x)[\np(x)+mu(x)]r(x) dx,
    \\
    &= \sum_{n=1}^\infty\frac{1}{\lambda_n^\varepsilon - \lambda}, \qquad \lambda^\varepsilon_n \in \sigma(T_{\varepsilon}).
    \end{align*}
    
    Denote by $\lambda_n^{\varepsilon, D} \in \sigma(T_{\varepsilon, D})$ the eigenvalues of the Dirichlet (i.e., Friedrichs) extension corresponding to $\tau|_{(x_\varepsilon,c)}$. Then we know that $\lambda^\varepsilon_n \leq \lambda_n^{\varepsilon,D}$. If we for simplicity assume that $\lambda < \lambda_n^\varepsilon$ for all $n$ (only finitely many $\lambda_n^\varepsilon$ could violate this), we conclude from above that  
    \begin{align*}
    \textrm{Tr}_{L^2((x_\varepsilon,c);r(x)dx)}\big((T_\varepsilon-\lambda I)^{-1}\big) \geq \sum_{n=1}^\infty\frac{1}{\lambda_n^{\varepsilon,D} - \lambda}, \qquad \lambda^{\varepsilon, D}_n \in \sigma(T_{\varepsilon, D}).
\end{align*}
Now, let $\lbrace \lambda_n \rbrace_{n = 1}^\infty = \sigma(T_F)$ be the eigenvalues of the Friedrichs extension of $\tau$. We can disregard any continuous spectrum of $T_F$, as its presence would already violate Hypothesis \ref{HypoGZ} which we assume. Then it follows from \cite[Thm.~10.8.2]{Ze05} that $\lim_{\varepsilon \to 0} \lambda_n^{\varepsilon,D} = \lambda_n$ for all $n \in \N_+$. As $(\lambda_n^{\varepsilon, D} - \lambda)^{-1} > 0$ by assumption, it follows through an application of Fatou's Lemma that
\begin{align*}
    \liminf_{\varepsilon \to 0} \int_{x_\varepsilon}^c u(x)[\np(x)+mu(x)]r(x) dx &\geq \liminf_{\varepsilon \to 0} \sum_{n=1}^\infty\frac{1}{\lambda_n^{\varepsilon,D} - \lambda}\geq \sum_{n=1}^\infty \frac{1}{\lambda_n - \lambda}.
\end{align*}
Hence if Hypothesis \ref{Hypothesis} holds at $x=a$ for $\lambda$, the infinite sum on the right must converge, implying Hypothesis \ref{hypothesisTrace} holds. This finishes the proof.
\end{proof}

\medskip

\noindent{\bf Acknowledgments.} We are thankful to Fritz Gesztesy and Gerald Teschl for fruitful discussions regarding the topics of the present manuscript, and to Benjamin Eichinger for bringing the papers \cite{LW23} and \cite{WW14} to our attention. We are also thankful for funding from The Ohio State University (OSU) for M.P.'s visit to OSU where a portion of this project was completed. M.P. was supported by the Methusalem grant METH/21/03 – long term structural funding of the Flemish Government.

\end{document}